\documentclass[10pt]{article}
\usepackage{amssymb,amsmath}
\usepackage{latexsym}
\numberwithin{equation}{section}

\newtheorem{theorem+}           {Theorem}      [section]
\newtheorem{definition+} {Definition}      [section]
\newtheorem{lemma+}  {Lemma}  [section]
\newtheorem{corollary+}  {Corollary} [section]
\newtheorem{proposition+}  {Proposition} [section]
\newtheorem{example+}  {Example}  [section]
\newtheorem{remark+}  {Remark}  [section]
\newtheorem{problem+}  {Problem}  [section]

\newenvironment{theorem}{\begin{theorem+}\sl}{\end{theorem+}\rm}
\newenvironment{problem}{\begin{problem+}\sl}{\end{problem+}\rm}
\newenvironment{definition}{\begin{definition+}\rm}{\end{definition+}\rm}
\newenvironment{lemma}{\begin{lemma+}\sl}{\end{lemma+}\rm}
\newenvironment{corollary}{\begin{corollary+}\sl}{\end{corollary+}\rm}
\newenvironment{proposition}{\begin{proposition+}\sl}{\end{proposition+}\rm}
\newenvironment{example}{\begin{example+}\rm}{\end{example+}\rm}
\newenvironment{remark}{\begin{remark+}\rm}{\end{remark+}\rm}
\newenvironment{proof}{\medbreak\noindent{\it Proof.\ }\rm}{\hfill$\square$\medbreak\rm}

\usepackage{mystyle}

\begin{document}

\begin{center}
{\Large\bf Necessary and sufficient conditions for the solvability
of the Gauss variational problem for infinite dimensional vector
measures}
\end{center}

\begin{center}
{\large Natalia Zorii}
\end{center}

\begin{abstract}
We continue our investigation of the Gauss variational problem for
infinite dimensional vector measures associated with a condenser
$(A_i)_{i\in I}$. It has been shown in Potential Anal.,
DOI:10.1007/s11118-012-9279-8 that, if some of the plates
(say~$A_\ell$ for $\ell\in L$) are noncompact then, in general,
there exists a vector~$\mathbf a=(a_i)_{i\in I}$, prescribing the
total charges on~$A_i$, $i\in I$, such that the problem admits {\it
no\/} solution. Then, what is a description of all the
vectors~$\mathbf a$ for which the Gauss variational problem is
nevertheless solvable? Such a characterization is obtained for a
positive definite kernel satisfying Fuglede's condition of
perfectness; it is given in terms of a solution to an auxiliary
extremal problem intimately related to the operator of orthogonal
projection onto the cone of all positive scalar measures supported
by~$\bigcup_{\ell\in L}\,A_\ell$. The results are illustrated by
examples pertaining to the Riesz kernels.

\vskip.2cm {\sl Subject classification}: 31C15.

\vskip.1cm {\sl Key words:} infinite dimensional vector measure,
external field, minimal energy problem
\end{abstract}

\section{Introduction}\label{sec:intr}

The interest in minimal energy problems in the presence of an
external field, which goes back to the pioneering works by
Gauss~\cite{Gauss} and Frostman~\cite{Fr}, was initially motivated
by their direct relations with the Dirichlet and balayage problems.
A new impulse to this part of potential theory (which is often
referred to as the Gauss variational problem) appeared in the 1980's
when Gonchar and Rakhmanov~\cite{GR0,GR1}, Mhaskar and
Saff~\cite{MaS} efficiently applied logarithmic potentials with
external fields in the investigation of orthogonal polynomials and
rational approximations to analytic functions. E.g., the vector
setting of the problem, earlier suggested by
Ohtsuka~\cite[Section~2.9]{O}, has nowadays become particularly
interesting in connection with Hermite--Pad\'{e} rational
approximations (see~\cite{A,GR0,GR,GRS,NS,ST}).

However, the potential-theoretical methods applied in these studies
were mainly based on the vague ({}=weak$^*$) topology, which made it
possible to establish the existence of a solution only for vector
measures of finite dimensions and compact support, though with a
rather general matrix of interaction between their components. See,
e.g., \cite[Theorem~2.30]{O}.

In order to treat the Gauss variational problem for vector measures
$\boldsymbol\mu=(\mu^i)_{i\in I}$ of infinite dimensions and/or
noncompact support, associated with a condenser, in~\cite{ZPot2} we
have suggested an approach based on introducing a metric structure
on the class of all~$\boldsymbol\mu$ with finite energy, which
agrees properly with the vague topology, and also on establishing an
infinite dimensional version of a completeness theorem (see below
for details).

This enabled us to obtain sufficient conditions for the solvability
of the problem (see~\cite[Theorem~8.1]{ZPot2}); moreover, we have
shown these sufficient conditions to be sharp by providing examples
of the nonsolvability (see~\cite[Examples~8.1,~8.2]{ZPot2}; see
below for some details).

It is worth emphasizing that, as is seen from these examples, such a
phenomenon of the nonsolvability occurs even for the Coulomb kernel
$|x-y|^{-1}$ in~$\mathbb R^3$ and a standard condenser of two
oppositely signed plates, one of them being noncompact, which at
first glance looks to be rather surprising because of an
electrostatic interpretation of the problem.

The purpose of this study is to establish conditions that are
simultaneously necessary and sufficient for the solvability of the
problem in the infinite dimensional vector setting.

To formulate the problem and to outline some of the results
obtained, we start by introducing relevant notions. In all that
follows, $\mathrm X$ denotes a locally compact Hausdorff space and
$\mathfrak M=\mathfrak M(\mathrm X)$ the linear space of all
real-valued scalar Radon measures~$\nu$ on~$\mathrm X$ equipped with
the {\it vague\/} topology, i.e. the topology of pointwise
convergence on the class $\mathrm C_0(\mathrm X)$ of all continuous
functions~$\varphi$ on~$\mathrm X$ with compact
support.\footnote{When speaking of a continuous function, we
understand that the values are {\it finite real\/} numbers.}

A {\it kernel\/}~$\kappa$ on $\mathrm X$ is meant to be an element
from $\mathrm\Phi(\mathrm X\times\mathrm X)$, where
$\mathrm\Phi(\mathrm Y)$ consists of all lower semicontinuous
functions $\psi:\mathrm Y\to(-\infty,\infty]$ such that
$\psi\geqslant0$ unless $\mathrm Y$ is compact. Given
$\nu,\nu_1\in\mathfrak M$, the {\it mutual energy\/} and the {\it
potential\/} relative to the kernel~$\kappa$ are defined by
$$\kappa(\nu,\nu_1):=\int\kappa(x,y)\,d(\nu\otimes\nu_1)(x,y)\quad\mbox{and}\quad
\kappa(\cdot,\nu):=\int\kappa(\cdot,y)\,d\nu(y),$$ respectively.
(When introducing a notation, we always tacitly assume the
corresponding object on the right to be well defined --- as a finite
number or $\pm\infty$.) For $\nu=\nu_1$, $\kappa(\nu,\nu_1)$ defines
the {\it energy\/} of~$\nu$. Let $\mathcal E=\mathcal E_\kappa$
consist of all $\nu\in\mathfrak M$ with
$-\infty<\kappa(\nu,\nu)<\infty$.

We shall mainly be concerned with a {\it positive definite\/}
kernel~$\kappa$, which by~\cite{F1} means that it is symmetric
(i.e., $\kappa(x,y)=\kappa(y,x)$ for all $x,y\in\mathrm X$) and
$\kappa(\nu,\nu)$, $\nu\in\mathfrak M$, is nonnegative whenever
defined. Then $\mathcal E$ forms a pre-Hil\-bert space with the
scalar product $\kappa(\nu,\nu_1)$ and the seminorm
$\|\nu\|_\mathcal E:=\|\nu\|_\kappa:=\sqrt{\kappa(\nu,\nu)}$. The
topology on~$\mathcal E$ defined by this seminorm is called {\it
strong\/}. A positive definite kernel~$\kappa$ is {\it strictly
positive definite\/} if the seminorm $\|\cdot\|_{\mathcal E}$ is a
norm.

Recall that a measure $\nu\geqslant0$ is {\it concentrated\/} on a
set $E\subset\mathrm X$ if $E^c:=\mathrm X\setminus E$ is locally
$\nu$-negligible; or equivalently, if $E$ is $\nu$-measurable and
$\nu=\nu|_E$, where $\nu|_E$ is the trace of~$\nu$ on~$E$ (see,
e.g., \cite{B2,E2}). Let $\mathfrak M^+_E$ be the convex cone of all
nonnegative measures concentrated on~$E$, and let $\mathcal
E^+_E:=\mathfrak M^+_E\cap\mathcal E$. We also write $\mathfrak
M^+:=\mathfrak M^+_\mathrm X$ and $\mathcal E^+:=\mathcal
E^+_\mathrm X$.

The {\it interior capacity\/} of a set~$E$ relative to the
kernel~$\kappa$ is given by the formula\footnote{Throughout the
paper, the infimum over the empty set is taken to be~$+\infty$.}
\begin{equation*}C(E):=C_\kappa(E):=1\bigl/\inf\,\kappa(\nu,\nu),\end{equation*}
where the infimum is taken over all $\nu\in\mathcal E^+_E$ with
$\nu(E)=1$.

We consider a countable, locally finite collection\footnote{If $I$
is a singleton, then we preserve the normal fonts instead of the
bold ones.} $\boldsymbol{A}=(A_i)_{i\in I}$ of closed sets, called
plates, $A_i\subset\mathrm X$ with the sign~$+1$ or $-1$ prescribed
such that the oppositely signed sets are mutually disjoint. Let
$\mathfrak M^+(\boldsymbol{A})$ stand for the Cartesian product
$\prod_{i\in I}\,\mathfrak M^+(A_i)$; then
$\boldsymbol\mu\in\mathfrak M^+(\boldsymbol{A})$ is a (nonnegative)
{\it vector measure\/} $(\mu^i)_{i\in I}$ with the components
$\mu^i\in\mathfrak M^+(A_i)$. The topology of the product space
$\prod_{i\in I}\,\mathfrak M^+(A_i)$, where every $\mathfrak
M^+(A_i)$ is endowed with the vague topology, is likewise called
{\it vague\/}. If $\boldsymbol{\mu}\in\mathfrak M^+(\boldsymbol{A})$
and a vector-valued function $\boldsymbol{u}=(u_i)_{i\in I}$ with
the $\mu^i$-measurable components $u_i:A_i\to[-\infty,\infty]$ are
given, then we write
$\langle\boldsymbol{u},\boldsymbol{\mu}\rangle:=\sum_{i\in I}\,\int
u_i\,d\mu^i$.\footnote{Here and in the sequel, an expression
$\sum_{i\in I}\,c_i$ is meant to be well defined provided that so is
every summand~$c_i$ and the sum does not depend on the order of
summation --- though might be $\pm\infty$.  Then, by the Riemann
series theorem, the sum is finite if and only if the series
converges absolutely.}

In accordance with an electrostatic interpretation of a condenser,
we assume that the interaction between the charges lying on the
conductors~$A_i$, $i\in I$, is characterized by the matrix
$(\alpha_i\alpha_j)_{i,j\in I}$, where $\alpha_i:={\rm sign}\,A_i$.
Given $\boldsymbol\mu,\boldsymbol\mu_1\in\mathfrak
M^+(\boldsymbol{A})$, we define the {\it mutual energy\/}
\begin{equation}\label{vectoren}\kappa(\boldsymbol\mu,\boldsymbol\mu_1):=\sum_{i,j\in
I}\,\alpha_i\alpha_j\kappa(\mu^i,\mu_1^j)\end{equation} and the {\it
vector potential\/} $\kappa_{\boldsymbol{\mu}}(x)$, $x\in\mathrm X$,
as a vector-valued function with the components
\begin{equation}\label{vectorpot}\kappa^i_{\boldsymbol{\mu}}(x):=\sum_{j\in
I}\,\alpha_i\alpha_j\kappa(x,\mu^j),\quad i\in I.\end{equation} For
$\boldsymbol\mu=\boldsymbol\mu_1$, the mutual energy
$\kappa(\boldsymbol\mu,\boldsymbol\mu_1)$ defines the {\it energy\/}
of~$\boldsymbol\mu$. Let $\mathcal E^+(\boldsymbol{A})$ consist of
all $\boldsymbol\mu\in\mathfrak M^+(\boldsymbol{A})$ with
$-\infty<\kappa(\boldsymbol\mu,\boldsymbol\mu)<\infty$.

Fix a vector-valued function $\boldsymbol{f}=(f_i)_{i\in I}$, where
each $f_i:\mathrm X\to[-\infty,\infty]$ is treated as an external
field acting on the charges on~$A_i$. Then the $\boldsymbol{f}$-{\it
weight\-ed vector potential\/} and the $\boldsymbol{f}$-{\it
weighted energy\/} of $\boldsymbol\mu\in\mathcal
E^+(\boldsymbol{A})$ are defined by
\begin{align}\label{def:wpot}\boldsymbol{W}_{\boldsymbol{\mu}}&:=\kappa_{\boldsymbol{\mu}}+\boldsymbol{f},\\
\label{def:wener}G_{\boldsymbol{f}}(\boldsymbol\mu)&:=
\kappa(\boldsymbol\mu,\boldsymbol\mu)+2\langle\boldsymbol{f},\boldsymbol\mu\rangle,\end{align}
respectively. Throughout this paper, we assume that\footnote{The
results obtained and the approach developed in~\cite{ZPot2} also
hold provided that $f_i\in\mathrm\Phi(\mathrm X)$ for all $i\in I$
(which is, in particular, the case if each~$f_i$ is the potential of
a {\it Dirac measure\/}). However, for a finer analysis to be
carried out in the present study in order to establish necessary and
sufficient conditions of the solvability, we need to assume the
external field to be generated--- in the sense of~(\ref{f})--- by a
charge of {\it finite\/} energy.}
\begin{equation}\label{f}f_i(x)=\alpha_i\kappa(x,\chi)\quad\text{for all \ } i\in
I,\end{equation} where a (signed) measure $\chi\in\mathcal E$ is
given, and we also write
$G_{\chi}(\boldsymbol\mu):=G_{\boldsymbol{f}}(\boldsymbol\mu)$.

Also fix a numerical vector $\boldsymbol{a}=(a_i)_{i\in I}$ with
$a_i>0$ for all $i\in I$ such that
\begin{equation}\label{ag}|\boldsymbol a|:=\sum_{i\in I}\,a_i<\infty\end{equation}
and a continuous function~$g$ on $A:=\bigcup_{i\in I}\,A_i$ with
\begin{equation}\label{ginf}g_{\inf}:=\inf_{x\in
A}\,g(x)>0.\end{equation} We are interested in the problem of
minimizing $G_{\chi}(\boldsymbol\mu)$ over the class of all
$\boldsymbol\mu\in\mathcal E^+(\boldsymbol{A})$ normalized by
$\langle g,\mu^i\rangle=a_i$ for all $i\in I$, referred to as the
{\it Gauss variational problem\/}.

The main question is whether minimizers $\boldsymbol\lambda$ in the
problem exist. If the condenser~$\boldsymbol{A}$ is finite and
compact, the kernel~$\kappa$ is continuous on $A_\ell\times A_j$
whenever $\alpha_\ell\ne\alpha_j$, and $f_i\in\mathrm\Phi(\mathrm
X)$ for all $i\in I$, then the existence of
those~$\boldsymbol\lambda$ can easily be established by exploiting
the vague topology only, since then the class of admissible vector
measures is vaguely compact, while
$G_{\boldsymbol{f}}(\boldsymbol\mu)$ is vaguely lower
semicontinuous. See~\cite[Theorem~2.30]{O}
(cf.~also~\cite{GR0,GR,NS,ST} for the logarithmic kernel in the
plane).

However, these arguments break down if any of the above-mentioned
four assumptions is dropped, and then the problem on the existence
of minimizers becomes rather nontrivial. In particular, the class of
admissible vector measures is no longer vaguely compact if any of
the~$A_i$ is noncompact. Another difficulty is that, under
assumption~(\ref{f}), the $\boldsymbol{f}$-weighted energy
$G_{\boldsymbol{f}}(\boldsymbol\mu)$ might not be vaguely lower
semicontinuous.

For a positive definite kernel satisfying Fuglede's condition of
consistency between the strong and vague topology on~$\mathcal E^+$
(see Definition~\ref{def:const} below), these difficulties have been
overcome in~\cite{ZPot2} in the frame of an approach\footnote{For a
background of this approach, see the pioneering work by
Fuglede~\cite{F1} (where $I=\{1\}$, $g=1$, and $f=0$) and also the
author's studies \cite{Z5a}--\cite{Z-Pot} (where $I$ is finite).}
based on the following crucial arguments.

The set $\mathcal E^+(\boldsymbol{A})$ has been shown to be a
semimetric space with the semimetric
\begin{equation}\label{vseminorm}
\|\boldsymbol\mu_1-\boldsymbol\mu_2\|_{\mathcal
E^+(\boldsymbol{A})}:=\Bigl[\sum_{i,j\in
I}\,\alpha_i\alpha_j\kappa(\mu^i_1-\mu^i_2,\mu^j_1-\mu^j_2)\Bigr]^{1/2},
\end{equation}
and one can define an inclusion $R$ of $\mathcal
E^+(\boldsymbol{A})$ into $\mathcal E$ such that $\mathcal
E^+(\boldsymbol{A})$ becomes isometric to its $R$-image, the latter
being regarded as a semimetric subspace of the pre-Hilbert
space~$\mathcal E$. See~\cite[Theorem~3.1]{ZPot2}; cf.~also
Theorem~\ref{lemma:semimetric} below. We therefore call the topology
of the semimetric space $\mathcal E^+(\boldsymbol{A})$ {\it
strong\/}.

Another crucial fact is that if, in addition, the kernel $\kappa$ is
bounded from above on the product of the oppositely signed plates,
then the topological subspace of~$\mathcal E^+(\boldsymbol{A})$
consisting of all~$\boldsymbol\mu$ with $\sum_{i\in
I}\,\mu^i(\mathrm X)\leqslant b$ (where $b$ is given) is strongly
complete. Moreover, the strong topology on this subspace is in
agreement with the vague one in the sense that any strong Cauchy net
converges strongly to each of its vague cluster points.
See~\cite[Theorem~9.1]{ZPot2}; cf.~also Theorem~\ref{th:complete}
below.

All these enabled us to prove in \cite{ZPot2} that if, moreover, $g$
is bounded from above, then the Gauss variational problem is
solvable provided that each $A_i$ either is compact or has finite
interior capacity. See~\cite[Theorem~8.1]{ZPot2}; cf.~also
Theorem~\ref{th:suff} below.

However, if some of the plates, say~$A_\ell$ for $\ell\in L$, are
noncompact and have infinite interior capacity then, in general,
there exists a vector~$\boldsymbol a$ such that the problem admits
{\it no\/} solution. See Examples~8.1 and~8.2 in~\cite{ZPot2},
illustrating such a phenomenon of the nonsolvability (cf.~also
Example~\ref{ex:sharp} below; see also~\cite{HWZ2,OWZ} for some
related numerical experiments).

Then, for given $\kappa$, $\boldsymbol A$, $g$, and $\chi$,  what is
a description of the set $\mathcal S_\kappa(\boldsymbol A,g,\chi)$
of all vectors~$\boldsymbol a$ for which the Gauss variational
problem is nevertheless solvable?

In this paper, such a characterization is established; it is given
in terms of a solution to an auxiliary extremal problem intimately
related to the operator of orthogonal projection onto the cone of
all positive scalar measures supported by~$\bigcup_{\ell\in
L}\,A_\ell$. In the case where $L$ is a singleton, the description
obtained is {\it complete}.

To give a hint how it looks like, we start with the following
example, where $L=\{\ell\}$ while $\kappa$ is the Riesz kernel
$\kappa_\alpha(x,y)=|x-y|^{\alpha-n}$ of order $\alpha\in(0,2]$,
$\alpha<n$, in~$\mathbb R^n$, $n\geqslant2$. For this kernel, the
operator of orthogonal projection in~$\mathcal E_{\kappa_\alpha}$
onto~$\mathcal E^+_{A_\ell}$ is, in fact, the operator of Riesz
balayage~$\beta^{\kappa_\alpha}_{A_\ell}$ onto~$A_\ell$, related to
the notion of $\alpha$-Green function~$g^\alpha_{A_\ell^c}$
of~$A_\ell^c$ by the formula
\begin{equation*}\label{greenn}g^\alpha_{A_\ell^c}(x,y):=
\kappa_\alpha(x,\varepsilon_y)-\kappa_\alpha(x,\beta^{\kappa_\alpha}_{A_\ell}\varepsilon_y),\end{equation*}
$\varepsilon_y$ being the unit Dirac measure at~$y$ (see, e.g.,
\cite[Chapter~4, Section~5]{L}).

\begin{example}\label{ex:11} Under these hypotheses, let the Euclidean
distance between the oppositely signed plates be nonzero, $A_i\cap
A_\ell=\varnothing$ for all $i\ne\ell$,
$C_{\kappa_\alpha}(A_i)>\delta>0$ for all $i\in I$, and let
$\chi\in\mathcal E_{\kappa_\alpha}$ be compactly supported in~$A^c$.

\begin{proposition}\label{pr:ex:11} Then the set\/
$\mathcal S_{\kappa_\alpha}(\boldsymbol A,g,\chi)$ consists of all
vectors\/ $\boldsymbol a=(a_i)_{i\in I}$ such that\/\footnote{Note
that $\sum_{i\ne\ell}\,\alpha_i\sigma^i$ is a scalar Radon measure
due to the local finiteness of~$\boldsymbol A$, so that
formula~(\ref{bal}) makes sense.}
\begin{equation}\label{bal}a_\ell\leqslant\Bigl\langle g,\beta^{\kappa_\alpha}_{A_\ell}\Bigl(\chi+\sum_{i\ne\ell}\,
\alpha_i\sigma^i\Bigr)\Bigr\rangle,\end{equation} where\/
$(\sigma^i)_{i\ne\ell}$ is a solution\/ {\rm(}it exists{\rm)} to the
problem of minimizing the $\alpha$-Green energy
\[\Bigl\|\chi+\sum_{i\ne\ell}\,
\alpha_i\mu^i\Bigr\|^2_{g^\alpha_{A_\ell^c}}\] over all\/
$(\mu^i)_{i\ne\ell}$ with the properties that\/ $\mu^i$, $i\ne\ell$,
is supported by\/~$A_i$ and\/ $\langle
g,\mu^i\rangle=a_i$.\end{proposition}\end{example}

The rest of the paper is organized as follows. In
Sections~\ref{sec:cond} and~\ref{sec:gausspr} we summarize without
proof some results from~\cite{ZPot2}, necessary for the
understanding of the subject matter. A description of the set
$\mathcal S_\kappa(\boldsymbol A,g,\chi)$ is provided by
Theorems~\ref{th:solvable1} and~\ref{th:solvable2}, the main results
of the study; it is given in terms of a
solution~$\tilde{\boldsymbol\lambda}$ to an auxiliary extremal
problem (see Section~\ref{sec:gausspr} for its formulation), while
the existence of this~$\tilde{\boldsymbol\lambda}$ is guaranteed by
Theorem~\ref{lemma:exist}. The proofs of Theorems~\ref{lemma:exist},
\ref{th:solvable1} and~\ref{th:solvable2} are provided in
Sections~\ref{proof:lemma:exist}, \ref{proof:th:solvable1}
and~\ref{proof:th:solvable2}, respectively; see also
Sections~\ref{sec:prel}, \ref{sec:J}, \ref{sec:extr}
and~\ref{sec:descr} for some crucial auxiliary results. Finally,
Section~\ref{sec:ex} contains some further examples with the Riesz
kernels, illustrating the results obtained.

\section{Preliminaries}

We write $S(\nu)$ or $S_\nu$ for the support of $\nu\in\mathfrak M$.
A measure~$\nu$ is called {\it finite\/} if  $S_\nu$ is compact, and
{\it bounded\/} if $|\nu|(\mathrm X)<\infty$. Here
$|\nu|:=\nu^++\nu^-$, where $\nu^+$ and $\nu^-$ are respectively the
positive and negative parts in the Hahn--Jordan decomposition
of~$\nu$.

In this section we assume the kernel~$\kappa$ to be positive
definite. Then $\mathcal E$ forms a pre-Hil\-bert space with the
scalar product $\kappa(\nu,\nu_1)$ and the seminorm
$\|\nu\|:=\|\nu\|_{\mathcal
E}:=\|\nu\|_\kappa:=\sqrt{\kappa(\nu,\nu)}$, while the potential
$\kappa(\cdot,\nu)$ of any $\nu\in\mathcal E$ is well defined and
finite {\it nearly everywhere\/} (n.e.) in~$\mathrm X$, i.e. except
for a subset with interior capacity zero; see~\cite{F1}.

In minimal energy problems, the following lemma from~\cite{F1} is
often useful.

\begin{lemma}\label{lemma:convex}Consider a convex set\/ $\mathcal H\subset\mathcal E$
and assume there exists\/ $\lambda_{\mathcal H}\in\mathcal H$ such
that
\[\|\lambda_{\mathcal H}\|=\min_{\nu\in\mathcal H}\,\|\nu\|.\]
Then
\[\|\nu-\lambda_{\mathcal H}\|^2\leqslant
\|\nu\|^2-\|\lambda_{\mathcal H}\|^2\quad\text{for all \
}\nu\in\mathcal H.\]
\end{lemma}

\subsection{Consistent and perfect
kernels}\label{sec:2}

In addition to the {\it strong\/} topology on~$\mathcal E$,
determined by the seminorm~$\|\nu\|$, sometimes we shall consider
the {\it weak\/} topology on~$\mathcal E$, defined by means of the
seminorms $\nu\mapsto|\kappa(\nu,\nu_1)|$, $\nu_1\in\mathcal E$
(see, e.g.,~\cite{F1}). The Cauchy--Schwarz inequality
\begin{equation*}
|\kappa(\nu,\nu_1)|\leqslant\|\nu\|\,\|\nu_1\|,\quad\mbox{where \ }
\nu,\nu_1\in\mathcal E,\end{equation*} implies immediately that the
strong topology on $\mathcal E$ is finer than the weak one.

In~\cite{F1,F2}, Fuglede introduced the following two {\it
equivalent\/} properties of consistency between the induced strong,
weak, and vague topologies on~$\mathcal E^+$:
\begin{itemize}
\item[\rm(C$_1$)] {\it Every strong Cauchy net in
$\mathcal E^+$ converges strongly to any of its vague cluster
points;}
\item[\rm(C$_2$)] {\it Every strongly bounded and vaguely convergent net in
$\mathcal E^+$ converges weakly to the vague limit.}
\end{itemize}

\begin{definition}\label{def:const}
Following Fuglede~\cite{F1}, we call a kernel~$\kappa$ {\it
consistent} if it satisfies either of the properties~(C$_1$)
and~(C$_2$), and {\it perfect\/} if, in addition, it is strictly
positive definite.\end{definition}

\begin{remark}\label{metriz} One has to consider {\it nets\/} or {\it filters\/}
in~$\mathfrak M^+$ instead of sequences, since the vague topology in
general does not satisfy the first axiom of countability. We follow
Moore's and Smith's theory of convergence, based on the concept of
nets (see~\cite{MS}; cf.~also~\cite[Chapter~0]{E2} and
\cite[Chapter~2]{K}). However, if $\mathrm X$ is metrizable and can
be written as a countable union of compact sets, then $\mathfrak
M^+$ satisfies the first axiom of countability
(see~\cite[Lemma~1.2.1]{F1}) and the use of nets may be
avoided.\end{remark}

\begin{theorem}{\rm(Fuglede \cite{F1})}\label{th:1} A kernel\/ $\kappa$ is perfect if and
only if\/ $\mathcal E^+$ is strongly complete and the strong
topology on\/~$\mathcal E^+$ is finer than the vague
one.\end{theorem}

\begin{remark} In $\mathbb R^n$, $n\geqslant 3$, the
Newtonian kernel $|x-y|^{2-n}$ is perfect~\cite{Car}. So are the
Riesz kernels $|x-y|^{\alpha-n}$, $0<\alpha<n$, in~$\mathbb R^n$,
$n\geqslant2$~\cite{D1,D2}, and the restriction of the logarithmic
kernel $-\log\,|x-y|$ in~$\mathbb R^2$ to the open unit
disk~\cite{L}. Furthermore, if $D$ is an open set in~$\mathbb R^n$,
$n\geqslant 2$, and its generalized Green function~$g_D$ exists
(see, e.g.,~\cite[Theorem~5.24]{HK}), then the kernel~$g_D$ is
perfect as well~\cite{E1}.\end{remark}

It is seen from Definition~\ref{def:const} and Theorem~\ref{th:1}
that the concept of consistent or perfect kernels is an efficient
tool in minimal energy problems over classes of {\it nonnegative
scalar\/} Radon measures with finite energy. Indeed, Fuglede's
theory of capacities of {\it sets\/} has been developed in~\cite{F1}
for exactly those kernels. The following fundamental result of this
theory will often be used below.\footnote{Compare
with~\cite[Theorem~10.1]{Z-Pot}, generalizing Theorem~\ref{thF:1} to
the interior capacities of condensers with finitely many plates.}

\begin{theorem}\label{thF:1} Assume that the kernel\/ $\kappa$ is consistent. Then
for any set\/ $E\subset\mathrm X$ with $C(E)<\infty$ there exists a
measure\/ $\theta_E\in\mathcal E^+_{\overline{E}}$ with the
properties
\begin{align}
\label{equF:1}\theta_E(\mathrm X)&=\|\theta_E\|^2=C(E),\\
\label{equ:F2}
\kappa(x,\theta_E)&\geqslant1\quad\text{n.e.~in \ }E,\\
\kappa(x,\theta_E)&\leqslant1\quad\text{for all \ }x\in
S(\theta_E).\notag\end{align} This\/ $\theta_E$, called an interior
equilibrium measure associated with\/~$E$, is a solution to the
problem of minimizing\/~$\kappa(\nu,\nu)$ over the
class\/~$\Gamma_E$ of all\/~$\nu\in\mathcal E$ such that\/
$\kappa(x,\nu)\geqslant1$ n.e.~in\/~$E$, and it is determined
uniquely up to a summand with seminorm zero.
\end{theorem}

\begin{remark} In~\cite{Z-Pot,ZPot2} we have shown that the concept of consistent or perfect kernels
is efficient, as well, in minimal energy problems over classes of
{\it vector measures\/} of finite or infinite dimensions, associated
with a condenser~$\boldsymbol A$. This is guaranteed by a theorem on
the completeness of certain topological subspaces of the semimetric
space~$\mathcal E^+(\boldsymbol{A})$ (see~\cite[Theorem~13.1]{Z-Pot}
and \cite[Theorem~9.1]{ZPot2}, cf.~also Theorem~\ref{th:complete}
below; compare with Theorem~\ref{th:1}).\end{remark}

\section{Auxiliary results related to scalar measures}\label{sec:prel}

In the following Lemmas~\ref{Lemma2.3.1}--\ref{lemma:vague.cont},
the kernel is arbitrary (not necessarily positive definite).

\begin{lemma}{\rm(Fuglede \cite{F1})}\label{Lemma2.3.1}
For any given\/ $E\subset\mathrm X$, it holds that\/
$C(E)=0\Longleftrightarrow\mathcal E^+_E=\varnothing$.\end{lemma}

\begin{lemma}\label{a.e.}If $\nu\in\mathcal E^+_{E}$ is bounded, then any proposition
holds\/ $\nu$-almost everywhere {\rm(}$\nu$-a.e.\/{\rm)}
in\/~$\mathrm X$, provided that it holds n.e.~in\/~$E$.\end{lemma}

\begin{proof}Indeed, then $E$ is $\nu$-integrable and hence,
by~\cite[Proposition~4.14.1]{E2}, any locally $\nu$-neg\-ligible
subset of~$E$ is $\nu$-negligible. Since, according to
Lemma~\ref{Lemma2.3.1}, a set of interior capacity zero is locally
$\xi$-negligible for any $\xi\in\mathcal E^+$, the lemma
follows.
\end{proof}

\subsection{On continuity of potentials}\label{sec:cont}

We shall need the following lemmas on continuity, the first being
well known (see, e.g.,~\cite{F1}).

\begin{lemma}\label{lemma:lower}
For any\/ $\psi\in\mathrm\Phi(\mathrm X)$ the map\/
$\nu\mapsto\langle\psi,\nu\rangle$ is vaguely lower semicontinuous
on\/~$\mathfrak M^+$.\end{lemma}

In particular, this implies that the potential $\kappa(\cdot,\nu)$
of any $\nu\in\mathfrak M^+$ belongs to~$\mathrm\Phi(\mathrm X)$.

\begin{lemma}\label{lemma:vague.cont2} Assume that\/ $\nu\in\mathfrak M^+$ is
bounded, $\kappa(x,y)$ is continuous for\/ $x\ne y$ and
\begin{equation}\label{kern.bound} \sup_{x\in K,\,y\in
S_\nu}\,\kappa(x,y)<\infty\quad\text{for every compact \ }K\subset
S^c_\nu.\end{equation}  Then the potential\/ $\kappa(\cdot,\nu)$ is
continuous at every\/ $x_0\notin S_\nu$.
\end{lemma}

\begin{proof} Having fixed a point $x_0\notin S_\nu$
and its compact neighborhood $V_{x_0}$ so that $V_{x_0}\cap
S_\nu=\varnothing$, we consider the function $\kappa^*(x,y)$ on
$V_{x_0}\times S_\nu$ defined by
\begin{equation}\label{kappastar}
\kappa^*(x,y):=-\kappa(x,y)+\sup_{{x'}\in V_{x_0},\,{y'}\in
S_\nu}\,\kappa({x'},{y'}).
\end{equation}
Under the hypotheses of the lemma, $\kappa^*$ is nonnegative and
continuous; hence,
\[\kappa^*(x,\nu)=\int\kappa^*(x,y)\,d\nu(y),\quad x\in
V_{x_0},\] is lower semicontinuous as the potential of
$\nu\in\mathfrak M^+$ relative to the kernel~$\kappa^*$.

On the other hand, integrating~(\ref{kappastar}) with respect to the
(bounded) measure~$\nu$, we conclude from
assumption~(\ref{kern.bound}) that $\kappa^*(x,\nu)$, $x\in
V_{x_0}$, coincides up to a finite summand with the restriction
of~$-\kappa(x,\nu)$ to~$V_{x_0}$. What has been shown just above
therefore implies that $-\kappa(\cdot,\nu)$ is lower semicontinuous
at~$x_0$, and the lemma follows.
\end{proof}

\begin{definition}\label{inf}A kernel $\kappa$ is said to {\it possess the property\/}
$(\infty_\mathrm X)$ if $\kappa(\cdot,y)\to0$ as $y\to\infty_\mathrm
X$ uniformly on compact sets; then for every $\varepsilon>0$ and
every compact $K\subset\mathrm X$ there exists a compact set
$K'\subset\mathrm X$ such that $|\kappa(x,y)|<\varepsilon$ for all
$x\in K$ and $y\in\mathrm X\setminus K'$.\end{definition}

For any $b\in(0,\infty)$, write $\mathfrak
M_b:=\bigl\{\nu\in\mathfrak M: \ |\nu|(\mathrm X)\leqslant
b\bigr\}$.

\begin{lemma}\label{lemma:vague.cont} Fix a closed set\/ $F\subset\mathrm
X$, a closed subset\/~$Q$ of\/~$F^c$, and\/ $b\in(0,\infty)$. If a
kernel\/ $\kappa(x,y)$ is continuous for\/ $x\ne y$ and possesses
the property\/~$(\infty_\mathrm X)$, then the mapping
\[(x,\nu)\mapsto\kappa(x,\nu)\quad\text{on \ }Q\times\bigl(\mathfrak
M^+_F\cap\mathfrak M_b\bigr)\]  is continuous in the product
topology, where\/ $Q$ and\/ $\mathfrak M^+_F\cap\mathfrak M_b$ are
considered to be topological subspaces of\/~$\mathrm X$
and\/~$\mathfrak M$, respectively.\footnote{Compare with~\cite[Lemma
2.2.1, assertion~(b)]{F1}.}
\end{lemma}

\begin{proof}Fix $x_0,x_s\in Q$ and $\nu_0,\nu_s\in\mathfrak
M^+_F\cap\mathfrak M_b$, $s\in S$, such that
$(x_s,\nu_s)\to(x_0,\nu_0)$ (as $s$ ranges along~$S$) in the
topology of the product space $Q\times\bigl(\mathfrak
M^+_F\cap\mathfrak M_b\bigr)$. We need to show that
\begin{equation}\label{cont1}\kappa(x_0,\nu_0)=\lim_{s\in
S}\,\kappa(x_s,\nu_s).\end{equation}

Due to the property $(\infty_\mathrm X)$, for any  $\varepsilon>0$
one can choose a compact neighborhood~$W_{x_0}$ of the point~$x_0$
and a compact neighborhood~$W$ of the set~$W_{x_0}$ so that
$W_{x_0}\cap F=\varnothing$ and
\begin{equation}\label{115}
|\kappa(x,y)|<\varepsilon b^{-1}\quad\text{for all \ }(x,y)\in
W_{x_0}\times W^c.\end{equation}  Certainly, there is no loss of
generality in assuming $V:=W\cap F\ne\varnothing$.

Let $E^{c_F}$ and~$\partial_FE$ denote respectively the complement
and the boundary of~$E$ relative to~$F$, where $F$ is treated as a
topological subspace of~$\mathrm X$. Having observed that
$\kappa|_{W_{x_0}\times F}$ is continuous, we proceed by
constructing a function $\varphi\in\mathrm C_0(W_{x_0}\times F)$
such that
\begin{equation}
\varphi|_{W_{x_0}\times V}=\kappa|_{W_{x_0}\times V},\label{rest}
\end{equation}
\begin{equation}
\left|\varphi(x,y)\right|\leqslant\varepsilon b^{-1}\quad\mbox{for
all \ } (x,y)\in W_{x_0}\times V^{c_F}.\label{118}
\end{equation}

To this end, consider a compact neighborhood~$V_*$ of~$V$ in~$F$ and
write \[ g:=\left\{
\begin{array}{cl} \kappa & \mbox{ \ on \ }
W_{x_0}\times\partial_FV,\\ 0 & \mbox{ \ on \ }
W_{x_0}\times\partial_FV_*.\\ \end{array} \right. \] Note that
$K:=(W_{x_0}\times\partial_FV)\cup (W_{x_0}\times\partial_FV_*)$ is
a compact subset of the Hausdorff and compact (hence, normal) space
$W_{x_0}\times V_*$, while the function~$g$ is continuous on~$K$. By
using the Tietze--Urysohn extension theorem (see, e.g.,
\cite[Chapter~0]{E2}), we deduce from relation~(\ref{115}) that
there exists a continuous function $\tilde{g}: \ W_{x_0}\times
V_*\to[-\varepsilon b^{-1},\varepsilon b^{-1}]$ such that
$\tilde{g}|_K=g|_K$. Hence, the function in question can be defined
by means of the formula \[ \varphi:=\left\{
\begin{array}{cl} \kappa & \mbox{ \ on \ }
W_{x_0}\times V,\\
\tilde{g} & \mbox{ \ on \ } W_{x_0}\times(V_*\setminus V),\\ 0 &
\mbox{ \ on \ } W_{x_0}\times V_*^{c_F}.
\end{array}
\right.\]

Furthermore, since such a function $\varphi$ is continuous on
$W_{x_0}\times F$ and has compact support, one can choose a compact
neighborhood~$U_{x_0}$ of~$x_0$ in~$W_{x_0}$ so that
\begin{equation}
\left|\varphi(x,y)-\varphi(x_0,y)\right|<\varepsilon
b^{-1}\quad\mbox{for all \ } (x,y)\in U_{x_0}\times F.\label{119}
\end{equation}
Therefore, for any $\nu\in\mathfrak M^+_F\cap\mathfrak M_b$ and
$x\in U_{x_0}$ we get, due to
relations~\mbox{(\ref{115})--(\ref{119})},
\begin{equation}
|\kappa(x,\nu|_{W^c})|\leqslant\varepsilon,\label{121}
\end{equation}
\begin{equation}
\kappa(x,\nu|_{W})=\int\varphi(x,y)\,d(\nu-\nu|_{W^c})(y),\label{122}
\end{equation}
\begin{equation}
\Bigl|\int\varphi(x,y)\,d\nu|_{W^c}(y)\Bigr|\leqslant\varepsilon,\label{123}
\end{equation}
\begin{equation}
\Bigl|\int
\bigl[\varphi(x,y)-\varphi(x_0,y)\bigr]\,d\nu(y)\Bigr|\leqslant
\varepsilon.\label{124}
\end{equation}

Choose $s_0\in S$ so that for all $s\in S$ that follow~$s_0$ there
hold $x_s\in U_{x_0}$ and
\[\Bigl|\int\varphi(x_0,y)\,d(\nu_s-\nu_0)(y)\Bigr|<\varepsilon.\]
Combined with relations (\ref{121})--(\ref{124}), this shows that
for all $s\geqslant s_0$,
\begin{equation*}
\begin{split}
|\kappa(x_s,\nu_s)-\kappa(x_0,\nu_0)|&\leqslant|\kappa(x_s,\nu_s|_{W})-
\kappa(x_0,\nu_0|_{W})|+2\varepsilon\\
{}&\leqslant\Bigl|\int\varphi(x_s,y)\,d\nu_s(y)-
\int\varphi(x_0,y)\,d\nu_0(y)\Bigr|+4\varepsilon\\
{}&\leqslant\Bigl|\int\bigl[\varphi(x_s,y)-\varphi(x_0,y)\bigr]\,d\nu_s(y)\Bigr|+\Bigl|\int
\varphi(x_0,y)\,d(\nu_s-\nu_0)(y)\Bigr|+4\varepsilon\\
{}&\leqslant\varepsilon+\varepsilon+4\varepsilon=6\varepsilon,
\end{split}
\end{equation*}
which in view of the arbitrary choice of $\varepsilon$
establishes~(\ref{cont1}).
\end{proof}

\subsection{Orthogonal projections in $\mathcal E$}\label{sec:proj}

Throughout this section we require the kernel $\kappa$ to be perfect
(see Definition~\ref{def:const}).

Fix $\nu\in\mathcal E$ and a closed set $F\subset\mathrm X$ with
$C(F)>0$, and let $P_F$ be the operator of {\it orthogonal
projection\/} in the pre-Hilbert space~$\mathcal E$ onto the convex
cone~$\mathcal E^+_F$ (see \cite[Section~1.12.3]{E2}; note that
$\mathcal E^+_F\ne\varnothing$ due to Lemma~\ref{Lemma2.3.1}). Then
$P_F\nu$ is a measure in $\mathcal E^+_F$ such that
\[\|\nu-P_F\nu\|=
\inf_{\omega\in\mathcal E^+_F}\,\|\nu-\omega\|=:\varrho(\nu,\mathcal
E^+_F).\] Observe that $\mathcal E^+_F$, treated as a metric
subspace of~$\mathcal E$, is complete in consequence
of~Theorem~\ref{th:1}; hence, according
to~\cite[Theorem~1.12.3]{E2}, $P_F\nu$ {\it exists and is determined
uniquely}.

\begin{lemma}\label{lemma:hryu} If\/
$K$ ranges through the increasing filtering family\/ $\{K\}_F$ of
all compact subsets of\/~$F$, then
\begin{equation*}P_K\nu\to P_F\nu\quad\text{strongly and vaguely}.\end{equation*}
\end{lemma}

\begin{proof} For each $K\in\{K\}_F$ one can certainly assume
$C(K)>0$, which causes no loss of generality because of
$C(F)=\sup_{K\subset F}\,C(K)$ (see~\cite[p.~153]{F1}); hence, the
projection~$P_K\nu$ exists (and is unique). We next observe that
$\varrho(\nu,\mathcal E^+_K)$ decreases as $K\uparrow F$ and
\begin{equation}\label{preg}\varrho(\nu,\mathcal E^+_F)\leqslant
\lim_{K\uparrow F}\,\varrho(\nu,\mathcal E^+_K).\end{equation} On
the other hand, applying \cite[Lemma~1.2.2]{F1} to the measure
$P_F\nu\otimes P_F\nu$ and the function~$\kappa$, as well as to
$P_F\nu$ and each of~$\kappa(\cdot,\nu^+)$
and~$\kappa(\cdot,\nu^-)$, we obtain
\[\|P_F\nu\|=\lim_{K\uparrow
F}\,\|(P_F\nu)|_K\|\quad\text{and}\quad\kappa(\nu,P_F\nu)=\lim_{K\uparrow
F}\,\kappa(\nu,(P_F\nu)|_K),\]  therefore
\[\varrho(\nu,\mathcal E^+_F)=\|\nu-P_F\nu\|=\lim_{K\uparrow F}\,\|\nu-(P_F\nu)|_K\|\geqslant
\lim_{K\uparrow F}\,\varrho(\nu,\mathcal E^+_K).\] When combined
with relation~(\ref{preg}), this establishes the equality
\begin{equation}\label{preg1}\varrho(\nu,\mathcal E^+_F)=\lim_{K\uparrow F}\,\varrho(\nu,\mathcal E^+_K).
\end{equation}

Fix $K,\widehat{K}\in\{K\}_F$ such that $K\subset\widehat K$.
Applying Lemma~\ref{lemma:convex} to the (convex) set
\[\nu-\mathcal E^+_{\widehat K}:=\Bigl\{\nu-\omega: \
\omega\in\mathcal E^+_{\widehat K}\Bigr\},\] in view of
$\nu-P_K\nu\in\nu-\mathcal E^+_{\widehat K}$  we get
\[\bigl\|P_{\widehat K}\nu-P_K\nu\bigr\|^2=\bigl\|(\nu-P_{\widehat K}\nu)-
(\nu-P_K\nu)\bigr\|^2\leqslant\varrho^2(\nu,\mathcal
E^+_K)-\varrho^2(\nu,\mathcal E^+_{\widehat K}).\] As
$\varrho(\nu,\mathcal E^+_K)$, $K\in\{K\}_F$, is fundamental
in~$\mathbb R$ because of~(\ref{preg1}), the last relation shows
that $(P_K\nu)_{K\in\{K\}_F}$ is a strong Cauchy net in~$\mathcal
E^+_F$. Since $\mathcal E^+_F$ is strongly complete, this net
converges to some $\omega_0\in\mathcal E^+_F$ strongly and, hence,
weakly. Repeated application of (\ref{preg1}) then yields
\[\lim_{K\uparrow F}\,\varrho(\nu,\mathcal E^+_K)=\lim_{K\uparrow F}\,\|\nu-P_K\nu\|=\|\nu-\omega_0\|=
\varrho(\nu,\mathcal E^+_F),\] which due to the uniqueness statement
gives $\omega_0=P_F\nu$, and the lemma follows.
\end{proof}

As an immediate consequence of Lemmas~\ref{lemma:lower}
and~\ref{lemma:hryu}, we get
\begin{align}
\label{proj:masses:lower}P_F\nu(\mathrm
X)&\leqslant\liminf_{K\uparrow F}P_K\nu(\mathrm X),\\
\label{proj:pot:lower}\kappa(x,P_F\nu)&\leqslant\liminf_{K\uparrow
F}\,\kappa(x,P_K\nu)\quad\text{n.e.~in \ }\mathrm X.\end{align}

\begin{lemma}\label{lemma:proj.desc} It holds that
\begin{equation}\label{orth1}
\kappa(x,P_F\nu)\geqslant\kappa(x,\nu)\quad\text{n.e.~in \
}F.\end{equation} If, moreover, $\kappa(x,y)$ is continuous for\/
$x\ne y$, $\nu^+$ is bounded, $S(\nu^+)\cap F=\varnothing$ and
\begin{equation*}\sup_{x\in K,\,y\in
S(\nu^+)}\,\kappa(x,y)<\infty\quad\text{for every compact \ }
K\subset F,\end{equation*} then
\begin{equation}\label{orth2}\kappa(x,P_F\nu)\leqslant\kappa(x,\nu)\quad\text{for
all \ }x\in S(P_F\nu)
\end{equation}
and therefore
\begin{equation}\label{orth3}
\kappa(x,P_F\nu)=\kappa(x,\nu)\quad \text{n.e.~in \
}S(P_F\nu).\end{equation}
\end{lemma}

\begin{proof}
According to \cite[Proposition~1.12.4]{E2}, $P_F\nu$ is uniquely
characterized by the relations
\begin{align}\label{orth4}
\kappa(\nu-P_F\nu,\omega)&\leqslant0\quad\text{for all \ }
\omega\in\mathcal E^+_F,\\
\label{orth5}\kappa(\nu-P_F\nu,P_F\nu)&=0,
\end{align}
We observe that assertion~(\ref{orth1}) can be obtained from
inequality~(\ref{orth4}) with the help of arguments similar to those
in \cite[Proof of Theorem~4.16]{L}. Indeed, for each
$\omega\in\mathcal E^+_F$ the set
\begin{equation*}\label{q}E:=\Bigl\{x\in F:
\ \kappa(x,P_F\nu)<\kappa(x,\nu)\Bigr\}\end{equation*} is
$\omega$-measurable and, hence, one can consider $\omega|_E$, the
trace of~$\omega$ on~$E$. Since $\omega|_E$ is an element of
$\mathcal E^+_F$ as well, relation~(\ref{orth4}) gives
\[\int\kappa(x,\nu-P_F\nu)\,d\omega|_E(x)\leqslant0,\]
which, however, is possible only provided that $E$ is locally
$\omega$-negligible. In view of the arbitrary choice of
$\omega\in\mathcal E^+_F$, this together with Lemma~\ref{Lemma2.3.1}
yields $C(E)=0$ as desired.

Let now all the hypotheses of the second part of the lemma  be
satisfied. It follows from inequality~(\ref{proj:pot:lower}) that,
while proving assertion~(\ref{orth2}), one can assume $F$ to be
compact. Then, by Lemma~\ref{a.e.}, relation~(\ref{orth1}) holds
$P_F\nu$-a.e. in~$\mathrm X$. This implies
\[\kappa(x,\nu-P_F\nu)=0\quad
P_F\nu\text{-a.e.~in \ }\mathrm X,\] for if not, we would arrive at
a contradiction with~(\ref{orth5}) when integrating
inequality~(\ref{orth1}) with respect to~$P_F\nu$. In turn, the last
relation yields that for every $x\in S(P_F\nu)$ one can choose a net
$(x_s)_{s\in S}\subset F$ with the properties that $x_s\to x$ and
\[\kappa(x_s,\nu-P_F\nu)=0\quad\text{for all \ }s\in S.\]
Hence, assertion (\ref{orth2}) will be proved once we show that
$\kappa(x,\nu-P_F\nu)$ is upper semicontinuous on~$F$. As
$\kappa(x,\nu^-+P_F\nu)$ is lower semicontinuous on~$\mathrm X$, it
is enough to establish the continuity of $\kappa(x,\nu^+)$ on~$F$,
but this is a direct consequence of
Lemma~\ref{lemma:vague.cont2}.
\end{proof}

\begin{definition}{\rm(see, e.g., \cite{L})}\label{def:h} A kernel $\kappa$ satisfies the {\it generalized maximum principle\/}
with a constant $h$ if for every finite $\nu\in\mathfrak M^+$ with
the property
\[\sup_{x\in S(\nu)}\,\kappa(x,\nu)=:M<\infty\] one has
$\kappa(x,\nu)\leqslant hM$ for all $x\in\mathrm X$.
\end{definition}

\begin{lemma}\label{lemma:bound}Suppose that the kernel\/~$\kappa$ is\/ ${}\geqslant0$ and satisfies the
generalized maximum principle with a constant\/~$h$.\footnote{Then,
obviously, $h\geqslant 1$.} If, moreover, all the assumptions of
Lemma {\rm\ref{lemma:proj.desc}} hold true, then
\begin{equation*}
P_F\nu(\mathrm X)\leqslant h\nu^+(\mathrm X).\end{equation*}
\end{lemma}

\begin{proof}It follows from inequality~(\ref{proj:masses:lower}) that, while
proving the lemma, one can assume~$F$ to be compact; then
$C(F)<\infty$ ($\kappa$ being strictly positive definite). Hence,
according to Theorem~\ref{thF:1}, an equilibrium measure
$\theta=\theta_{S(P_F\nu)}$ of~$S(P_F\nu)$ exists and satisfies the
relations
\begin{align} \label{equ:1}
\kappa(x,\theta)&\geqslant1\quad\text{n.e.~in \ }S(P_F\nu),\\
\kappa(x,\theta)&\leqslant1\quad\text{for all \ }x\in
S(\theta).\label{equ:2}\end{align} Since then, due to the
boundedness of both~$P_F\nu$ and~$\theta$, relations~(\ref{orth3})
and (\ref{equ:1}) hold true $\theta$-a.e. and $P_F\nu$-a.e.,
respectively, on account of $\kappa\geqslant0$ we get
\begin{equation*}P_F\nu(\mathrm
X)\leqslant\int\kappa(x,\theta)\,dP_F\nu(x)=\int\kappa(x,\nu)\,d\theta(x)
\leqslant\int\kappa(x,\theta)\,d\nu^+(x).\end{equation*} On the
other hand, inequality~(\ref{equ:2}) yields, by
Definition~\ref{def:h},
\[\kappa(x,\theta)\leqslant h\quad\text{for all \ }x\in\mathrm X.\]
When substituted into the preceding
relation, this yields the lemma.
\end{proof}

\section{Vector measures associated with condensers; their energies and potentials. Strong completeness
theorem}\label{sec:cond}

Let $I^+$ and $I^-$ be fixed countable, disjoint sets of
indices~$i\in\mathbb N$, where the latter is allowed to be empty,
and let $I:=I^+\cup I^-$. Assume that to every $i\in I$ there
corresponds a nonempty, closed set~$A_i\subset\mathrm X$.

\begin{definition} A collection $\boldsymbol{A}=(A_i)_{i\in I}$ is
called an $(I^+,I^-)$-{\it condenser\/} (or simply a {\it
condenser\/}) in~$\mathrm X$ if every compact subset of~$\mathrm X$
intersects with at most finitely many~$A_i$ and
\begin{equation}
A_i\cap A_j=\varnothing\quad\mbox{for all \ } i\in I^+, \ j\in I^-.
\label{non}
\end{equation}\end{definition}

We call $\boldsymbol{A}$ {\it compact\/} if so are all~$A_i$, $i\in
I$, and {\it finite\/} if $I$ is finite. The sets $A_i$, $i\in I^+$,
and $A_j$, $j\in I^-$, are called the {\it positive\/} and {\it
negative plates}, respectively. (Note that any two equally sign\-ed
plates can intersect each other or even coincide.) For any
$I_0\subseteq I$, write
\[CI_0:=I\setminus I_0,\quad A_{I_0}:=\bigcup_{i\in I_0}\,A_i,\quad
A^+:=A_{I^+},\quad A^-:=A_{I^-},\quad A:=A_{I};\] then $A_{I_0}$ is
closed ($\boldsymbol{A}$ being locally finite), and it is compact if
$A_i$, $i\in I_0$, are compact while $I_0$ is finite.

Given $\boldsymbol{A}$, let $\mathfrak M^+(\boldsymbol{A})$ consist
of all (nonnegative) {\it vector measures\/}
$\boldsymbol\mu=(\mu^i)_{i\in I}$, where $\mu^i\in\mathfrak
M^+(A_i)$ for all $i\in I$. The product topology on~$\mathfrak
M^+(\boldsymbol{A})$, where every $\mathfrak M^+(A_i)$ is equipped
with the vague topology, is likewise called {\it vague\/}.

In accordance with an electrostatic interpretation of a condenser,
we assume that the law of interaction between the charges lying on
the plates~$A_i$, $i\in I$, is determined by the matrix
$(\alpha_i\alpha_j)_{i,j\in I}$, where $\alpha_i$ equals~$+1$ if
$i\in I^+$ and~$-1$ otherwise. Given
$\boldsymbol\mu,\boldsymbol\mu_1\in\mathfrak M^+(\boldsymbol{A})$,
we then define the {\it mutual energy\/}
$\kappa(\boldsymbol\mu,\boldsymbol\mu_1)$ and the {\it vector
potential\/} $\kappa_{\boldsymbol{\mu}}(x)$, $x\in\mathrm X$, by
relations~(\ref{vectoren}) and~(\ref{vectorpot}), respectively. For
$\boldsymbol\mu=\boldsymbol\mu_1$, the mutual energy
$\kappa(\boldsymbol\mu,\boldsymbol\mu_1)$ defines the {\it energy\/}
of~$\boldsymbol\mu$. Let $\mathcal E^+(\boldsymbol{A})$ consist of
all $\boldsymbol\mu\in\mathfrak M^+(\boldsymbol{A})$ with
$-\infty<\kappa(\boldsymbol\mu,\boldsymbol\mu)<\infty$.

From now on we shall always require the kernel $\kappa$ to be
positive definite. Then, according to~\cite[Lemma~3.3]{ZPot2}, for
$\boldsymbol\mu\in\mathfrak M^+(\boldsymbol{A})$ to have finite
energy, it is sufficient that $\sum_{i\in I}\,\|\mu^i\|<\infty$.

In this paper we are concerned with minimal energy problems over
subsets of~$\mathcal E^+(\boldsymbol{A})$, to be treated in the
frame of the approach developed in~\cite{ZPot2}. For the convenience
of the reader, in this and the next sections we summarize without
proof some results from~\cite{ZPot2}, necessary for the
understanding of the subject matter.

Due to the local finiteness of $\boldsymbol A$, one can define the
mapping $R:\mathfrak M^+(\boldsymbol{A})\to\mathfrak M$,
\[R\boldsymbol\mu(\varphi)=\sum_{i\in I}\,\alpha_i\mu^i(\varphi)\quad\mbox{for
all \ }\varphi\in\mathrm C_0(\mathrm X).\]  Such a mapping is, in
general, non-injective; it is injective if and only if $A_i$, $i\in
I$, are mutually disjoint. Also observe that $R\boldsymbol\mu$ is a
{\it signed\/} (scalar) measure; because of assumption~(\ref{non}),
the positive and negative parts in the Hahn--Jordan decomposition
of~$R\boldsymbol\mu$ are given by
\begin{equation}\label{HJ} R\boldsymbol\mu^+=\sum_{i\in
I^+}\,\mu^i\quad\text{and}\quad R\boldsymbol\mu^-=\sum_{i\in
I^-}\,\mu^i.\end{equation}

\begin{lemma}\label{lemma3.5}For any\/ $\psi\in\mathrm\Phi(\mathrm X)$ and\/
$\boldsymbol\mu\in\mathfrak M^+(\boldsymbol{A})$, the value\/
$\langle\psi,R\boldsymbol\mu\rangle$ is finite if and only if\/
$\sum_{i\in I}\,|\langle\psi,\mu^i\rangle|<\infty$, and then
\[\langle\psi,R\boldsymbol\mu\rangle=\sum_{i\in I}\,\alpha_i\langle\psi,\mu^i\rangle.\]
\end{lemma}

\begin{corollary}\label{lemma3.3}
$\boldsymbol\mu\in\mathfrak M^+(\boldsymbol{A})$ has finite energy
if and only if so does its scalar\/ $R$-image,
$R\boldsymbol\mu$.\footnote{This implies in particular that
$\mathcal E^+(\boldsymbol A)$ forms a {\it convex
cone}.}\end{corollary}

\begin{corollary}\label{relation}For any\/ $\boldsymbol\mu,\boldsymbol\mu_1\in\mathcal
E^+(\boldsymbol A)$, it holds that
\begin{equation*}\label{vecen}
\kappa(\boldsymbol\mu,\boldsymbol\mu_1)=\kappa(R\boldsymbol\mu,R\boldsymbol\mu_1).
\end{equation*}
Furthermore, $\kappa_{\boldsymbol\mu}(x)$ is well defined and finite
n.e.\footnote{We say that a collection of propositions $P_i(x)$,
$i\in I$, subsists n.e.~in~$\mathrm X$ if this is the case for
each~$P_i(x)$.} in\/~$\mathrm X$ and has the components
\begin{equation}\label{Rpot}
\kappa^i_{\boldsymbol\mu}(x)=\alpha_i\kappa(x,R\boldsymbol\mu),\quad
i\in I.
\end{equation}
\end{corollary}

\begin{theorem}\label{lemma:semimetric} $\mathcal E^+(\boldsymbol{A})$ forms a
semimetric space with the semimetric\/
$\|\boldsymbol\mu_1-\boldsymbol\mu_2\|_{\mathcal
E^+(\boldsymbol{A})}$, defined by\/~{\rm(\ref{vseminorm})}, and such
a space is isometric to its\/ $R$-image, $R\mathcal
E^+(\boldsymbol{A})$. That is,
\begin{equation}\label{twonorms}\|\boldsymbol\mu_1-\boldsymbol\mu_2\|_{\mathcal
E^+(\boldsymbol{A})}=\|R\boldsymbol\mu_1-R\boldsymbol\mu_2\|_{\mathcal
E}\quad\text{for all \ }\boldsymbol\mu_1,\boldsymbol\mu_2\in\mathcal
E^+(\boldsymbol{A}).\end{equation} The semimetric\/
$\|\boldsymbol\mu_1-\boldsymbol\mu_2\|_{\mathcal
E^+(\boldsymbol{A})}$ is a metric if and only if the kernel\/
$\kappa$ is strictly positive definite while all\/ $A_i$, $i\in I$,
are mutually disjoint.
\end{theorem}

\begin{corollary}\label{cor:EA}Given\/ $\boldsymbol A$, define\/ $\mathcal E_A$ as the set of
all\/ $\nu$ such that\/ $\nu^+\in\mathcal E^+_{A^+}$ and\/
$\nu^-\in\mathcal E^+_{A^-}$ and equip it with the semimetric
structure induced from\/~$\mathcal E$. Then
\begin{equation}\label{EA}\mathcal E_A=R\mathcal E^+(\boldsymbol A),\end{equation} so
that, by Theorem\/~{\rm\ref{lemma:semimetric}}, the semimetric
spaces\/ $\mathcal E_A$ and\/ $\mathcal E^+(\boldsymbol A)$ are
isometric.\end{corollary}

\begin{proof}Indeed, $R\mathcal E^+(\boldsymbol A)\subset\mathcal E_A$ can be
obtained directly from relation~(\ref{HJ}) and
Corollary~\ref{lemma3.3}. To prove the converse inclusion, fix
$\nu\in\mathcal E_A$ and define $\mu^1:=|\nu|\bigl|_{A_1}$ and
\[\mu^{i+1}:=\Biggr[|\nu|-\sum_{k=1}^i\,\mu^k\Biggr]\Biggl|_{A_{i+1}}
\quad\text{for all \ }i=1,2,\ldots.\] Then $\mu^i\in\mathfrak
M^+(A_i)$ and $\nu=R\boldsymbol{\mu}$, where
$\boldsymbol{\mu}:=(\mu^i)_{i\in I}$. Since $\kappa(\nu,\nu)$ is
finite, repeated application of Corollary~\ref{lemma3.3} shows that
$\boldsymbol{\mu}\in\mathcal E^+(\boldsymbol
A)$, and identity (\ref{EA}) follows.
\end{proof}

Because of (\ref{twonorms}), the topology on $\mathcal
E^+(\boldsymbol{A})$ defined by the semimetric
$\|{\boldsymbol\mu}_1-\boldsymbol\mu_2\|_{\mathcal
E^+(\boldsymbol{A})}$ is called {\it strong\/}. Two elements
of~$\mathcal E^+(\boldsymbol{A})$, $\boldsymbol\mu_1$
and~$\boldsymbol\mu_2$, are {\it equivalent\/} if
$\|\boldsymbol\mu_1-\boldsymbol\mu_2\|_{\mathcal
E^+(\boldsymbol{A})}=0$. The equivalence in~$\mathcal
E^+(\boldsymbol{A})$ implies $R\boldsymbol\mu_1=R\boldsymbol\mu_2$
provided that the kernel~$\kappa$ is strictly positive definite, and
it implies $\boldsymbol\mu_1=\boldsymbol\mu_2$ if, moreover, all
the~$A_i$ are mutually disjoint.

\begin{corollary}\label{lemma:potequiv}If\/ $\boldsymbol\mu_1$ and\/~$\boldsymbol\mu_2$ are equivalent in\/
$\mathcal E^+(\boldsymbol{A})$, then\/
$\kappa_{\boldsymbol\mu_1}(x)=\kappa_{\boldsymbol\mu_2}(x)$
n.e.~in\/~$\mathrm X$.
\end{corollary}

For any $b\in(0,\infty)$, let $\mathfrak M_{b}^+(\boldsymbol{A})$
consist of all $\boldsymbol{\mu}\in\mathfrak M^+(\boldsymbol{A})$
with
\[|R\boldsymbol{\mu}|(\mathrm
X)=\bigl(R\boldsymbol{\mu}^++R\boldsymbol{\mu}^-\bigr)(\mathrm
X)=\sum_{i\in I}\,\mu^i(\mathrm X)\leqslant b.\] This class is
vaguely bounded and closed; hence, by \cite[Lemma~3.1]{ZPot2}, it is
vaguely compact.

The following theorem on the strong completeness of $\mathcal
E_{b}^+(\boldsymbol{A}):=\mathcal E^+(\boldsymbol{A})\cap\mathfrak
M_{b}^+(\boldsymbol{A})$, treated as a topological subspace
of~$\mathcal E^+(\boldsymbol{A})$, is crucial in our approach
(compare with Theorem~\ref{th:1}).\footnote{In view of the fact that
the semimetric space $\mathcal E_{b}^+(\boldsymbol{A})$ is isometric
to its $R$-image, Theorem~\ref{th:complete} has singled out a {\it
strongly complete\/} topological subspace of the pre-Hilbert
space~$\mathcal E$, whose elements are {\it signed\/} Radon
measures. This is of independent interest since, according to a
well-known counterexample by Cartan~\cite{Car}, the whole
space~$\mathcal E$ is strongly incomplete even for the Newtonian
kernel $|x-y|^{2-n}$ in~$\mathbb R^n$, $n\geqslant3$.}

\begin{theorem}\label{th:complete}Assume that the kernel\/ $\kappa$ is
consistent and bounded on\/ $A^+\times A^-$.\footnote{For the Riesz
kernels $\kappa_\alpha$ of order $\alpha\in(0,n)$ in~$\mathbb R^n$,
$n\geqslant2$, Theorem~\ref{th:complete} remains valid  without
assuming~$\kappa_\alpha$ to be bounded on $A^+\times A^-$;
cf.~\cite[Theorem~1]{Z2} and~\cite[Remark~9.2]{ZPot2}. The proof is
based on Deny's theorem~\cite{D1} stating that the pre-Hilbert
space~$\mathcal E_{\kappa_\alpha}$ can be completed by making use of
distributions of finite energy.} Then the following two assertions
hold:
\begin{itemize}
\item[\rm(i)] The semimetric space\/ $\mathcal
E_{b}^+(\boldsymbol{A})$ is complete. In more detail, if\/
$(\boldsymbol\mu_s)_{s\in S}\subset\mathcal E^+_b(\boldsymbol{A})$
is a strong Cauchy net and\/ $\boldsymbol\mu$~is one of its vague
cluster points, then\/ $\boldsymbol\mu\in\mathcal
E^+_b(\boldsymbol{A})$ and $\boldsymbol\mu_s\to\boldsymbol\mu$
strongly, i.e.
\begin{equation*}
\lim_{s\in S}\,\|\boldsymbol\mu_s-\boldsymbol\|_{\mathcal
E^+(\boldsymbol A)}=0.
\end{equation*}
\item[\rm(ii)] If the kernel\/ $\kappa$ is strictly positive
definite\/ {\rm(}hence, perfect\/{\rm)}, while all\/ $A_i$, $i\in
I$, are mutually disjoint, then the strong topology on\/~$\mathcal
E^+_b(\boldsymbol{A})$ is finer than the vague one. In more detail,
if\/ $(\boldsymbol\mu_s)_{s\in S}\subset\mathcal
E^+_b(\boldsymbol{A})$ converges strongly to\/
$\boldsymbol\mu_0\in\mathcal E^+(\boldsymbol{A})$, then actually\/
$\boldsymbol\mu_0\in\mathcal E^+_b(\boldsymbol{A})$ and\/
$\boldsymbol\mu_s\to\boldsymbol\mu_0$
vaguely.\end{itemize}\end{theorem}

\section{Gauss variational problem}\label{sec:gausspr}

Fix $\chi\in\mathcal E$ and define $\boldsymbol{f}=(f_i)_{i\in I}$
where $f_i(x):=\alpha_i\kappa(x,\chi)$ for all $i\in I$; such
an~$\boldsymbol f$ is well defined and finite n.e.~in~$\mathrm X$,
while all its components are {\it universally measurable}, i.e.
$\nu$-meas\-urable for every $\nu\in\mathfrak M^+$. We treat
each~$f_i$ as an external field acting on the charges on~$A_i$ and
then we define the $\boldsymbol f$-weighted vector potential
$\boldsymbol W_{\boldsymbol\mu}=(W_{\boldsymbol\mu}^i)_{i\in I}$ and
the $\boldsymbol{f}$-{\it weighted energy\/}
$G_\chi(\boldsymbol{\mu}):=G_{\boldsymbol f}(\boldsymbol{\mu})$ of
$\boldsymbol\mu\in\mathcal E^+(\boldsymbol{A})$ by
relations~(\ref{def:wpot}) and~(\ref{def:wener}), respectively.

Given $\boldsymbol\mu\in\mathcal E^+(\boldsymbol{A})$, application
of Lemma~\ref{lemma3.5} to each of $\kappa(\cdot,\chi^+)$
and~$\kappa(\cdot,\chi^-)$ gives
\begin{equation*}\langle\boldsymbol
f,\boldsymbol{\mu}\rangle=\kappa(\chi,R\boldsymbol{\mu}).\end{equation*}
Combined with~(\ref{twonorms}), this implies that
$G_\chi(\boldsymbol{\mu})$ is finite and can be written in the form
\begin{equation}\label{gmu}G_{\chi}(\boldsymbol{\mu})=\|R\boldsymbol{\mu}\|^2+2\kappa(\chi,R\boldsymbol{\mu})=
-\|\chi\|^2+\|\chi+R\boldsymbol{\mu}\|^2.\end{equation}

Fix a vector $\boldsymbol{a}=(a_i)_{i\in I}$ with $a_i>0$ for all
$i\in I$ and a continuous function~$g$ on~$A$ satisfying
conditions~(\ref{ag}) and~(\ref{ginf}). Also having fixed~$J$ such
that $I^+\subseteq J\subseteq I$, we write
\[\mathcal
E^+(\boldsymbol{A},\boldsymbol{a},g,J):=\Bigl\{\boldsymbol{\mu}\in\mathcal
E^+(\boldsymbol{A}): \ \langle g,\mu^i\rangle=a_i\quad\text{for all
\ }i\in J\Bigr\}\] and further we consider
\[G_{\chi}(\boldsymbol{A},\boldsymbol{a},g,J):=\inf_{\boldsymbol{\mu}\in\mathcal
E^+(\boldsymbol{A},\boldsymbol{a},g,J)}\,G_{\chi}(\boldsymbol{\mu}).\]
Then, because of (\ref{gmu}),
\begin{equation}\label{ga}G_{\chi}(\boldsymbol{A},\boldsymbol{a},g,J)=-\|\chi\|^2+\inf_{\boldsymbol{\mu}\in\mathcal
E^+(\boldsymbol{A},\boldsymbol{a},g,J)}\,\|\chi+R\boldsymbol{\mu}\|^2.\end{equation}

In the case $J=I$, the symbol $J$ in these and other notations will
often be dropped. That is, we write $\mathcal
E^+(\boldsymbol{A},\boldsymbol{a},g):=\mathcal
E^+(\boldsymbol{A},\boldsymbol{a},g,I)$,
$G_{\chi}(\boldsymbol{A},\boldsymbol{a},g):=G_{\chi}(\boldsymbol{A},\boldsymbol{a},g,I)$,
and so on.

Below we use the following terminology. Given $\boldsymbol{A}$,
$\boldsymbol{a}$, $g$, $\chi$, and $J$, the
$(\boldsymbol{A},\boldsymbol{a},g,\chi,J)$-{\it prob\-lem\/} is that
on the existence of $\tilde{\boldsymbol{\lambda}}\in\mathcal
E^+(\boldsymbol{A},\boldsymbol{a},g,J)$ with
\[G_\chi(\tilde{\boldsymbol{\lambda}})=\inf_{\boldsymbol\mu\in\mathcal
E^+(\boldsymbol{A},\boldsymbol{a},g,J)}\,G_\chi(\boldsymbol\mu).\]
The $(\boldsymbol{A},\boldsymbol{a},g,\chi,J)$-problem is said to be
{\it solvable\/} if the class $\mathfrak
S_\chi(\boldsymbol{A},\boldsymbol{a},g,J)$ of all its minimizers is
nonempty. The $(\boldsymbol{A},\boldsymbol{a},g,\chi,I)$-problem (or
shortly the $(\boldsymbol{A},\boldsymbol{a},g,\chi)$-problem), main
in the present study, is referred to as {\it the Gauss variational
problem}. A minimizer in the
$(\boldsymbol{A},\boldsymbol{a},g,\chi)$-prob\-lem will often be
denoted simply by~$\boldsymbol\lambda$ (with the tilde dropped).

To make these problems well defined, we shall always suppose that
\begin{equation}\label{G:finite}G_\chi(\boldsymbol{A},\boldsymbol{a},g)<\infty;\footnote{Necessary and/or sufficient conditions for
assumption~(\ref{G:finite}) to hold can be found in~\cite{ZPot2}
(see Lemmas~6.2 and~6.3 therein). In particular, then necessarily
$C(A_i)>0$ for all $i\in I$. See also Remark~\ref{rem:fin}
below.}\end{equation} then for every $J$ one has
\begin{equation*}-\infty<G_\chi(\boldsymbol{A},\boldsymbol{a},g,J)<\infty,\end{equation*}
where the inequality on the right is an obvious consequence of
assumption~(\ref{G:finite}), while that on the left follows
from~(\ref{ga}) due to the positive definiteness of the kernel.

In addition to all the assumptions stated above, in the rest of the
paper  we shall always require that the kernel~$\kappa$ is
consistent and bounded on $A^+\times A^-$, i.e.
\begin{equation}\label{kern.bound''}\sup_{(x,y)\in A^+\times A^-}\,\kappa(x,y)<\infty,\end{equation}
while for every $i\in I$ either $g|_{A_i}$ is bounded or there exist
$r_i\in(1,\infty)$ and $\nu_i\in\mathcal E$ such that
\begin{equation}
g|_{A_i}^{r_i}(x)\leqslant\kappa(x,\nu_i)\quad\mbox{n.e.~in \ }
\mathrm X. \label{growth}
\end{equation}

Then sufficient conditions for the solvability of the Gauss
variational problem are given by the following theorem
(see~\cite[Theorem~8.1]{ZPot2}).

\begin{theorem}\label{th:suff}Under the above-mentioned hypotheses, the\/ $(\boldsymbol{A},\boldsymbol{a},g,\chi)$-problem is solvable
if each\/~$A_i$, $i\in I$, either is compact or has finite interior
capacity.
\end{theorem}

Moreover, this theorem is sharp; that is, if at least one of the
plates is noncompact and has infinite interior capacity, then the
$(\boldsymbol{A},\boldsymbol{a},g,\chi)$-problem, in general, admits
{\it no\/} solution. This can be illustrated by the following
example (cf.~\cite[Proposition~8.1]{ZPot2}).\footnote{See
also~\cite[Example~8.2]{ZPot2} pertaining to the Coulomb kernel
$|x-y|^{-1}$ in~$\mathbb R^3$, and also~\cite{HWZ2,OWZ} for some
related numerical experiments.}

\begin{example}\label{ex:sharp}In $\mathbb R^n$, $n\geqslant2$, consider the Riesz kernel
$\kappa_\alpha(x,y)=|x-y|^{\alpha-n}$ of order~$\alpha$, where
$\alpha\in(0,2]$ and $\alpha<n$. Assume $I^+=\{1\}$, $I^-=\{2\}$,
$A_1$ and $A_2$ are closed disjoint sets with
$C_{\kappa_\alpha}(A_i)\ne0$, $i=1,2$, and let $A_1$ be compact
while $A_2^c$ connected. If, moreover, $\chi\in\mathcal E^+$ is
compactly supported in~$A_2^c$ and $a_2=a_1+\chi(\mathbb R^n)$, then
the corresponding $(\boldsymbol{A},\boldsymbol{a},g,\chi)$-problem
has {\it no\/} solution if and only if
$C_{\kappa_\alpha}(A_2)=\infty$ though $A_2$ is $\alpha$-thin
at~$\infty_{\mathbb R^n}$.\footnote{A closed set $F\subset\mathbb
R^n$ is {\it $\alpha$-thin at\/}~$\infty_{\mathbb R^n}$ if $F^*$,
the inverse of~$F$ relative to the unit sphere, is $\alpha$-thin at
$x=0$, or equivalently \cite[Theorem~5.10]{L}, if either $F$ is
bounded or $x=0$ is an $\alpha$-irregular point for~$F^*$.
See~\cite{Z1}; for $\alpha=2$, such a definition is due to M.~Brelot
(see~\cite{Brelo2}; cf.~also~\cite{F3,Hayman2}).  We refer
to~\cite{Z0} for an example of a set of infinite Newtonian capacity,
though $2$-thin at~$\infty_{\mathbb R^n}$
(cf.~also~Example~\ref{ex:13} in Section~\ref{sec:ex}
below).}\end{example}

\begin{problem}\label{pr}Given $\kappa$, $\boldsymbol A=(A_i)_{i\in I}$, $g$, and $\chi$,  what is
a description of the set $\mathcal S_\kappa(\boldsymbol A,g,\chi)$
of all vectors $\boldsymbol a=(a_i)_{i\in I}$ for which the
$(\boldsymbol{A},\boldsymbol{a},g,\chi)$-problem is nevertheless
solvable?\end{problem}

\section{Standing assumptions}\label{sec:standing}

In addition to the standing assumptions stated in
Section~\ref{sec:gausspr}, in the rest of the paper we assume that
$\mathrm X$ is noncompact (hence, $\kappa\geqslant0$). This involves
no loss of generality, since for a compact~$\mathrm X$ the Gauss
variational problem is always solvable due to Theorem~\ref{th:suff}.

We also require that the kernel~$\kappa$ is strictly positive
definite (hence, perfect), continuous for $x\ne y$, possesses the
property~$(\infty_\mathrm X)$ and satisfies the generalized maximum
principle with a constant $h\geqslant 1$ (see Definitions~\ref{inf}
and~\ref{def:h}). Furthermore, assume that the measure
$\chi\in\mathcal E$, generating the external field by means
of~(\ref{f}), is bounded and satisfies the assumptions
\begin{equation*}\label{chi}S(\chi^+)\cap A^-=\varnothing,\quad S(\chi^-)\cap
A^+=\varnothing.\end{equation*}

\begin{remark}\label{rem:lower}Then for each $i\in I$
the $i$-component of the external field,
$f_i=\alpha_i\kappa(\cdot,\chi)$, is lower semicontinuous on~$A_i$,
which is seen from Lemma~\ref{lemma:vague.cont2}.\end{remark}

\begin{remark}\label{rem:fin}Suppose for a moment that $\chi\in\mathcal E$ is compactly supported in~$A^c$.
It is seen from~\cite[Lemma~6.3]{ZPot2} that, under the above
hypotheses on the kernel~$\kappa$, assumption~(\ref{G:finite}) would
hold automatically if one would require $C(A_i)>\delta>0$ for all
$i\in I$.\end{remark}

\begin{remark}
Note that the standing assumptions on a kernel are not too
restrictive. In particular, in~$\mathbb R^n$, $n\geqslant3$, they
all are satisfied by the Newtonian kernel $|x-y|^{2-n}$, Riesz
kernels $|x-y|^{\alpha-n}$ with $\alpha\in(0,n)$, or Green kernels
$g_D$, provided that the Euclidean distance between~$A^+$ and~$A^-$
is nonzero. Here $D\subset\mathbb R^n$ is a regular domain (in the
sense of the solvability of the classical Dirichlet problem) and
$g_D$ its Green function.
\end{remark}

\section{Description of the set $\mathcal S_\kappa(\boldsymbol A,g,\chi)$}\label{sec:main}

Recall that we are working under the standing assumptions stated in
Sections~\ref{sec:gausspr} and~\ref{sec:standing}.

Before having formulated an answer to Problem~\ref{pr}, we observe
that in the case $J\ne I$ the auxiliary $(\boldsymbol A,\boldsymbol
a,g,\chi,J)$-problem can equivalently be formulated in terms of
orthogonal projections~$P_{A_{CJ}}$ onto the convex cone\/~$\mathcal
E^+_{A_{CJ}}$ in the sense of the following lemma.\footnote{Note
that $\mathcal E^+_{A_{CJ}}\ne\varnothing$ due to the fact that
$C(A_i)>0$ for all $i\in I$ (see the footnote to
formula~(\ref{G:finite})), so that for any $\nu\in\mathcal E$ an
orthogonal projection $P_{A_{CJ}}\nu$ exists and is determined
uniquely (see~Section~\ref{sec:proj}).}

Given $\boldsymbol A$ and $J\ne I$, define $\boldsymbol
a_J:=(a_i)_{i\in J}$ and $\boldsymbol A_J:=(A_i)_{i\in J}$, the
latter being thought of as an $(I^+,I^-\cap J)$-condenser. For any
$\boldsymbol{\mu}=(\mu^i)_{i\in I}\in\mathfrak M^+(\boldsymbol A)$,
also write $\boldsymbol{\mu}_{J}:=(\mu^i)_{i\in J}$.

With these assumptions and notations, consider the minimum energy
problem
\begin{equation}\label{pr:pr}\inf_{\boldsymbol\nu\in\mathcal E^+(\boldsymbol A_J,\boldsymbol
a_J,g)}\,\bigl\|\chi+R\boldsymbol\nu-P_{A_{CJ}}(\chi+R\boldsymbol\nu)\bigr\|^2.\end{equation}

\begin{lemma}\label{aux1}For\/~$\tilde{\boldsymbol\lambda}$ to solve the\/ $(\boldsymbol
A,\boldsymbol a,g,\chi,J)$-problem with\/ $J\ne I$, it is necessary
and sufficient that there exist a solution\/~$\boldsymbol\sigma$ to
the problem\/~{\rm(\ref{pr:pr})} such that
\begin{equation}\label{relrel}\tilde{\boldsymbol\lambda}_J=\boldsymbol\sigma\quad\text{and}\quad
R\tilde{\boldsymbol\lambda}_{CJ}=P_{A_{CJ}}(\chi+R\boldsymbol\sigma).\end{equation}
\end{lemma}

\begin{proof} Fix $\boldsymbol{\mu}\in\mathcal
E^+(\boldsymbol{A},\boldsymbol{a},g,J)$ and let
$\boldsymbol\omega=(\omega^i)_{i\in CJ}$ be one of the $R$-preimages
of~$P_{A_{CJ}}(\chi+R\boldsymbol{\mu}_J)$. Consider
$\boldsymbol{\mu}'$ such that
$\boldsymbol{\mu}'_J=\boldsymbol{\mu}_J$ and
$\boldsymbol{\mu}'_{CJ}=\boldsymbol\omega$. Then, by
Corollaries~\ref{lemma3.3} and~\ref{cor:EA},
$\boldsymbol{\mu}'\in\mathcal E^+(\boldsymbol{A})$; actually,
$\boldsymbol{\mu}'\in\mathcal
E^+(\boldsymbol{A},\boldsymbol{a},g,J)$. Therefore, using
representation~(\ref{gmu}), we get
\begin{align*}\label{gchi2}G_\chi(\boldsymbol{\mu})&=-\|\chi\|^2+\|\chi+R\boldsymbol{\mu}\|^2\geqslant
-\|\chi\|^2+\|\chi+R\boldsymbol{\mu}_{J}-P_{A_{CJ}}(\chi+R\boldsymbol{\mu}_{J})\|^2\\{}&=
G_\chi(\boldsymbol{\mu}')\geqslant
G_\chi(\boldsymbol{A},\boldsymbol{a},g,J),\end{align*} where the
first inequality is an equality if and only if
$R\boldsymbol\mu_{CJ}= P_{A_{CJ}}(\chi+R\boldsymbol\mu_J)$. In view
of the arbitrary choice of $\boldsymbol{\mu}\in\mathcal
E^+(\boldsymbol{A},\boldsymbol{a},g,J)$,
this yields the lemma.
\end{proof}

\subsection{Main results}

Let $L$ consist of all $\ell\in I$ such that $C(A_\ell)=\infty$.
Assume $L\ne\varnothing$, for otherwise Problem~\ref{pr} has already
been solved by Theorem~\ref{th:suff}. In all that follows, we also
suppose that $L\subseteq I^-$.

A description of the set $\mathcal S_\kappa(\boldsymbol A,g,\chi)$,
required in Problem~\ref{pr}, is given by
Theorems~\ref{th:solvable1} and~\ref{th:solvable2}, the main results
of this study. It is formulated in terms of a solution to the
auxiliary $(\boldsymbol A,\boldsymbol a,g,\chi,CL)$-problem, while
the solvability of the latter is guarantied by the following theorem
with $J=CL$.

\begin{theorem}\label{lemma:exist}
Given\/ $J$, where\/ $I^+\subseteq J\subseteq I$, assume that
\begin{equation}\label{capfinite}
C(A_j)<\infty\quad\text{for all \ } j\in J.
\end{equation}
Then the\/ $(\boldsymbol A,\boldsymbol a,g,\chi,J)$-problem is
solvable, and the class\/ $\mathfrak S_\chi(\boldsymbol
A,\boldsymbol a,g,J)$ of all its solutions is vaguely compact.
\end{theorem}

Fix a solution~$\tilde{\boldsymbol\lambda}$ to the auxiliary
$(\boldsymbol A,\boldsymbol a,g,\chi,CL)$-prob\-lem. Then,
equivalently, $\tilde{\boldsymbol\lambda}_{CL}$ gives a solution to
the minimum energy problem~(\ref{pr:pr}) with $J=CL$, while
$\tilde{\boldsymbol\lambda}_L$ is one of the $R$-pre\-images of
$P_{A_L}(\chi+R\tilde{\boldsymbol\lambda}_{CL})$; see
Lemma~\ref{aux1}.

\begin{theorem}\label{th:solvable1} Consider\/ $\boldsymbol a=(a_i)_{i\in I}$ such
that\/\footnote{The values $\langle g,\tilde{\lambda}^\ell\rangle$,
$\ell\in L$, can be shown to be finite
(see~Section~\ref{proof:th:solvable1}), so that
inequalities~(\ref{br}) make sense.}
\begin{equation}\label{br}a_\ell\geqslant\langle
g,\tilde{\lambda}^\ell\rangle\quad\text{for all \ }\ell\in
L.\end{equation} Then
\begin{equation}\label{LLL}G_\chi(\boldsymbol A,\boldsymbol a,g)=G_\chi(\boldsymbol A,\boldsymbol a,g,CL),\end{equation}
while the main\/ $(\boldsymbol A,\boldsymbol a,g,\chi)$-problem is
solvable if and only if all the inequalities\/~{\rm (\ref{br})} are
equalities.
\end{theorem}

In the case where $L$ is a singleton, Theorem~\ref{th:solvable1} can
be refined as follows to establish a {\it complete description\/} of
the set $\mathcal S_\kappa(\boldsymbol A,g,\chi)$.

\begin{theorem}\label{th:solvable2} Let\/ $L=\{\ell\}$ and assume, moreover, that
\begin{equation}\label{nonint}A_i\cap
A_\ell=\varnothing\quad\text{for all \ }i\ne\ell.\end{equation} Then
the set\/ $\mathcal S_\kappa(\boldsymbol A,g,\chi)$ consists of all
vectors\/ $\boldsymbol a=(a_i)_{i\in I}$ with the property
\begin{equation*}\label{aell}
a_\ell\leqslant\langle g,\tilde{\lambda}^\ell\rangle
\end{equation*}
or, equivalently,
\begin{equation}\label{aell''}
a_\ell\leqslant\Bigl\langle
g,P_{A_\ell}\Bigl(\chi+\sum_{i\ne\ell}\,\alpha_i\tilde{\boldsymbol\lambda}^i\Bigr)\Bigr\rangle=:\mathrm\Sigma_\ell.
\end{equation}
\end{theorem}

\begin{remark} Theorem~\ref{th:solvable2} makes sense, since, under its hypotheses,
the value $\mathrm\Sigma_\ell=\langle g,\tilde{\lambda}^\ell\rangle$
does not depend on the choice of
$\tilde{\boldsymbol\lambda}\in\mathfrak
S_\chi(\boldsymbol{A},\boldsymbol{a},g,I\setminus\{\ell\})$
(see~Section~\ref{proof:th:solvable2}).\end{remark}

See Sections~\ref{proof:lemma:exist}, \ref{proof:th:solvable1}, and
\ref{proof:th:solvable2} for the proofs of
Theorems~\ref{lemma:exist}--\ref{th:solvable2}. Combining
Theorem~\ref{th:solvable1} with Lemma~\ref{lemma:H} for $J=CL$ (see
Section~\ref{sec:J} below), we arrive at the following corollary.

\begin{corollary}\label{cor:solvable1} Assume\/ $L$ is finite and\/ $g_{\sup}:=\sup_{x\in A}\,g(x)<\infty$.
Then the\/ $(\boldsymbol A,\boldsymbol a,g,\chi)$-prob\-lem is
nonsolvable for every\/~$\boldsymbol a$ with
\[a_\ell>hg_{\sup}\Bigl[\chi^+(\mathrm X)+2|\boldsymbol a_{CL}|g_{\inf}^{-1}\Bigr]\quad\text{for all \ }\ell\in L,\]
$h$ being the constant in the generalized maximum principle.
\end{corollary}

\subsection{Examples}\label{sec:ex}

In this section, we withdraw all the standing assumptions stated in
Sections~\ref{sec:gausspr} and~\ref{sec:standing}.

To illustrate the results obtained, we consider the Riesz kernel
$\kappa_\alpha(x,y)=|x-y|^{\alpha-n}$ of order $\alpha\in(0,2]$,
$\alpha<n$, in~$\mathbb R^n$, $n\geqslant2$, and we assume that
$A_\ell$, where $\ell\in I^-$, is the only plate with infinite
interior capacity. For this kernel, the operator of orthogonal
projection~$P_{A_\ell}$ is, in fact, the operator of Riesz
balayage~$\beta^{\kappa_\alpha}_{A_\ell}$ onto~$A_\ell$, while (see,
e.g.,~\cite{L})
\[\|\nu-\beta^{\kappa_\alpha}_{A_\ell}\nu\|_{\kappa_\alpha}=\|\nu\|_{g^\alpha_{A_\ell^c}}
\quad\text{for all \ }\nu\in\mathcal E_{\kappa_\alpha},\] where
$g^\alpha_{A_\ell^c}$ is the $\alpha$-Green function of the open set
$A_\ell^c$.

Hence, the auxiliary problem~(\ref{pr:pr}) with
$J=I\setminus\{\ell\}$ can be rewritten as
\begin{equation}\label{pr:pr'}\inf\,\|\chi+R\boldsymbol\nu\|_{g^\alpha_{A_\ell^c}}^2,\end{equation}
the infimum being taken over all $\boldsymbol\nu\in\mathcal
E_{\kappa_\alpha}^+(\boldsymbol A_{I\setminus\{\ell\}})$ such that
$\langle g,\nu^i\rangle=a_i$ for all $i\ne\ell$, while
$\mathrm\Sigma_\ell$ from formula~(\ref{aell''}) now takes the form
\begin{equation*}\label{sigma}\mathrm\Sigma_\ell=\bigl\langle
g,\beta^{\kappa_\alpha}_{A_\ell}(\chi+R\boldsymbol\sigma)\bigr\rangle,\end{equation*}
where $\boldsymbol\sigma$ is a solution to the minimum
$\alpha$-Green energy problem~(\ref{pr:pr'}).

Combined with Theorem~\ref{th:solvable2}, this yields
Proposition~\ref{pr:ex:11} (see Example~\ref{ex:11}), providing a
complete description of the set $\mathcal
S_{\kappa_\alpha}(\boldsymbol A,g,\chi)$. See also
Examples~\ref{ex:12}--\ref{ex:14} below, where such a description
takes a much simpler form, given in terms of the geometry of the
plate~$A_\ell$.

In each of the following examples, let
$C_{\kappa_\alpha}(A_i)>\delta>0$ for all $i\in I$ and let
$\chi\in\mathcal E_{\kappa_\alpha}$ be compactly supported
in~$A^c$.\footnote{Then assumption~(\ref{G:finite}) holds
automatically (cf.~Remark~\ref{rem:fin}).} We also require that
$\chi\geqslant0$, ${\rm dist}\,(A^+,A^-)>0$, $I^-=L=\{\ell\}$,
$A_\ell^c$ is connected, and $g=1$.

\begin{example}\label{ex:12} Under the hypotheses stated above, assume moreover that
\begin{equation}\label{init}2a_\ell\geqslant|\boldsymbol a|+\chi(\mathbb
R^n).\end{equation}

\begin{proposition}\label{pr:ex:12}
Then the Gauss variational problem\/ {\rm(}relative to\/
$\kappa_\alpha$, $\boldsymbol A$, $\boldsymbol a$, $\chi\geqslant0$,
and\/~$g=1${\rm)} is solvable if and only if\/  $A_{\ell}$ is not\/
$\alpha$-thin at\/~$\infty_{\mathbb R^n}$, while
\begin{equation}\label{init'}2a_\ell=|\boldsymbol a|+\chi(\mathbb
R^n).\end{equation}
\end{proposition}
\begin{proof} It is known from the Riesz balayage theory (see, e.g.,
\cite{L}) that, for any bounded positive measure $\nu\in\mathcal
E_{\kappa_\alpha}(A_\ell^c)$, it holds
$\beta^{\kappa_\alpha}_{A_\ell}\nu(\mathbb R^n)\leqslant\nu(\mathbb
R^n)$, while according to \cite[Theorem~4]{Z1} the inequality here
is an equality if and only if $A_\ell$ is not $\alpha$-thin
at~$\infty_{\mathbb R^n}$. Combining this for
$\nu=\chi+R\boldsymbol\sigma$ with Theorem~\ref{th:solvable2}, where
$\mathrm\Sigma_\ell$ is now given by
\begin{equation*}\mathrm\Sigma_\ell=\beta^{\kappa_\alpha}_{A_\ell}(\chi+R\boldsymbol\sigma)(\mathbb R^n),\end{equation*}
and taking assumption~(\ref{init}) into account, we obtain the
proposition.
\end{proof}
\end{example}

\begin{example}\label{ex:13}Consider the Coulomb kernel
$\kappa_2(x,y)=|x-y|^{-1}$ in~$\mathbb R^3$ and let $A_\ell$ be a
rotational body consisting of all $x=(x_1,x_2,x_3)\in\mathbb R^3$
such that $q\leqslant x_1<\infty$ and $0\leqslant
x_2^2+x_3^2\leqslant\rho(x_1)$, where $q\in\mathbb R$ and
$\rho(x_1)$ approaches~$0$ as $x_1\to\infty$. We focus with the
following three cases:
\begin{align}\rho(r)&=r^{-s},\phantom{\exp-}\quad\text{where \
}s\in[0,\infty),\label{ex1}\\
\rho(r)&=\exp(-r^s),\quad\text{where \ }s\in(0,1],\label{ex2}\\
\rho(r)&=\exp(-r^s),\quad\text{where \ }s>1.\label{ex3}\end{align}
Then $A_\ell$ is not $2$-thin at~$\infty_{\mathbb R^3}$ in
case~(\ref{ex1}), has finite (Newtonian) capacity in
case~(\ref{ex3}), and it is $2$-thin at~$\infty_{\mathbb R^3}$
though $C_{\kappa_2}(A_\ell)=\infty$ in case~(\ref{ex2});
see~\cite{Z0}. Hence, if~$\boldsymbol a$ satisfies
condition~(\ref{init}), then, by Proposition~\ref{pr:ex:12}, the
Gauss variational problem is nonsolvable in case~(\ref{ex2}), while
in case~(\ref{ex1}) it admits a solution if and only if
(\ref{init'}) additionally holds. Also observe that, by
Theorem~\ref{th:suff}, in case~(\ref{ex3}) the problem is solvable
for {\it any\/}~$\boldsymbol a$.
\end{example}

\begin{example}\label{ex:14}Consider $\boldsymbol A=(A_1,A_2)$ with
$\alpha_1=+1$ and $\alpha_2=-1$, and let $\chi=0$. Then, in
consequence of Theorem~\ref{th:solvable2}, the set $\mathcal
S_{\kappa_\alpha}(\boldsymbol A,g,\chi)$ forms the cone in~$\mathbb
R^2$ defined by
\begin{equation}\label{aa}
a_2/a_1\leqslant\bigl(\beta^{\kappa_\alpha}_{A_2}\theta_{A_1}\bigr)\bigl(A_2\bigr),
\end{equation}
where $\theta_{A_1}$ is obtained from the equilibrium measure
of~$A_1$ relative to the kernel~$g^\alpha_{A_2^c}$ when divided by
$C_{g^\alpha_{A_2^c}}(A_1)$. Note that $\theta_{A_1}(A_1)=1$; hence,
by~\cite[Theorem~4]{Z1}, the value on the right in
relation~(\ref{aa}) is equal to~$1$ if $A_2$ is not $\alpha$-thin
at~$\infty_{\mathbb R^n}$, while otherwise it is~${}<1$.
\end{example}

\section{On the admissible measures in the auxiliary $(\boldsymbol A,\boldsymbol
a,g,J,\chi)$-problem}\label{sec:J}

For any $b\in(0,\infty)$ and $J$, $I^+\subseteq J\subseteq I$, write
$\mathcal E_b^+(\boldsymbol{A},\boldsymbol{a},g,J):=\mathcal
E^+(\boldsymbol{A},\boldsymbol{a},g,J)\cap\mathcal
E_{b}^+(\boldsymbol{A})$, where $\mathcal E_{b}^+(\boldsymbol{A})$
has been defined in Section~\ref{sec:cond}.

\begin{lemma}\label{lemma:H}
The value\/ $G_\chi(\boldsymbol{A},\boldsymbol{a},g,J)$ remains
unchanged if the class\/ $\mathcal
E^+(\boldsymbol{A},\boldsymbol{a},g,J)$ in its definition is
replaced by\/ $\mathcal E_H^+(\boldsymbol{A},\boldsymbol{a},g,J)$,
where
\begin{equation}\label{h}H:=h\Bigl[\chi^+(\mathrm X)+2|\boldsymbol a_J|g_{\inf}^{-1}\Bigr].\end{equation}
\end{lemma}

\begin{proof} Observe that $H$ is finite, which is clear from assumptions~(\ref{ag}), (\ref{ginf}) and the boundedness
of~$\chi$. Also note that, for any $\boldsymbol{\mu}=(\mu^i)_{i\in
I}\in\mathcal E^+(\boldsymbol{A},\boldsymbol{a},g,J)$,
\begin{equation}\label{upper}\mu^j(\mathrm X)\leqslant
a_jg^{-1}_{\inf}<\infty\quad\text{for all \ }j\in J.\end{equation}
Hence, if $J=I$, then $\mathcal
E^+(\boldsymbol{A},\boldsymbol{a},g,I)\subset\mathcal
E_H^+(\boldsymbol{A},\boldsymbol{a},g,I)$ and the lemma is obvious.

Therefore, assume $J\ne I$. Then, as is clear from the proof of
Lemma~\ref{aux1}, $G_\chi(\boldsymbol{A},\boldsymbol{a},g,J)$
remains unchanged if the class $\mathcal
E^+(\boldsymbol{A},\boldsymbol{a},g,J)$ in its definition is
replaced by its subclass consisting of all
$\boldsymbol\mu\in\mathcal E^+(\boldsymbol{A},\boldsymbol{a},g,J)$
with the additional property that $R{\boldsymbol\mu}_{CJ}=
P_{A_{CJ}}(\chi+R{\boldsymbol\mu}_{J})$. Thus, the lemma will follow
once we prove
\begin{equation}\label{est1}R\boldsymbol\mu_J(\mathrm X)+P_{A_{CJ}}(\chi+R\boldsymbol{\mu}_J)(\mathrm
X)\leqslant H,\end{equation} but this is a direct consequence of
relation~(\ref{upper}) and Lemma~\ref{lemma:bound}.
\end{proof}

\begin{corollary}\label{min:proj}If\/ $\tilde{\boldsymbol\lambda}$ is a minimizer in the\/
$(\boldsymbol{A},\boldsymbol{a},g,\chi,J)$-problem, then\/
$\tilde{\boldsymbol\lambda}\in\mathcal
E_H^+(\boldsymbol{A},\boldsymbol{a},g,J)$.
\end{corollary}

\begin{proof}For $J\ne I$ this follows from
relations~(\ref{relrel}) and~(\ref{est1}), while otherwise it is
obvious.
\end{proof}

\section{Extremal measures}\label{sec:extr}

\begin{definition}\label{def:minnet}We call a net $(\boldsymbol\mu_s)_{s\in S}$ {\it minimizing\/} in the
$(\boldsymbol{A},\boldsymbol{a},g,\chi,J)$-problem if
\begin{equation*}\label{b}(\boldsymbol\mu_s)_{s\in S}\subset\mathcal
E_H^+(\boldsymbol{A},\boldsymbol{a},g,J),\end{equation*} $H$ being
defined by (\ref{h}), and furthermore
\begin{equation}\lim_{s\in
S}\,G_{\chi}(\boldsymbol\mu_s)=G_{\chi}(\boldsymbol{A},\boldsymbol{a},g,J).
\label{min}\end{equation}\end{definition}

Let $\mathbb M_\chi(\boldsymbol{A},\boldsymbol{a},g,J)$ consist of
all those nets; this class is nonempty in consequence of
assumption~(\ref{G:finite}) and Lemma~\ref{lemma:H}. We denote by
$\mathcal M_\chi(\boldsymbol{A},\boldsymbol{a},g,J)$ the union of
the vague cluster sets of $(\boldsymbol\mu_s)_{s\in S}$, where
$(\boldsymbol\mu_s)_{s\in S}$ ranges over $\mathbb
M_\chi(\boldsymbol{A},\boldsymbol{a},g,J)$.

\begin{remark} Taking a subnet if necessary, we shall always
assume $(\boldsymbol\mu_s)_{s\in S}\in\mathbb
M_\chi(\boldsymbol{A},\boldsymbol{a},g,J)$ to be strongly bounded.
Then so are $(R\boldsymbol\mu^+_s)_{s\in S}$,
$(R\boldsymbol\mu^-_s)_{s\in S}$ and $(\mu^i_s)_{s\in S}$ for all
$i\in I$; that is,
\begin{equation}\label{twostars}
\sup_{s\in S,\,i\in
I}\,\Bigl\{\|\mu_s^i\|,\,\|R\boldsymbol\mu^+_s\|,\,\|R\boldsymbol\mu^-_s\|\Bigr\}<\infty.
\end{equation}
Indeed, this can be obtained from $\kappa\geqslant0$ and
$\kappa(R\boldsymbol\mu^+_s,R\boldsymbol\mu^-_s)\leqslant M<\infty$
for all $s\in S$, the latter being a consequence of
$|R\boldsymbol\mu_s|(\mathrm X)\leqslant H$ and
assumption~(\ref{kern.bound''}).
\end{remark}

\begin{definition}\label{def:extr} We call
$\tilde{\boldsymbol\gamma}=(\tilde{\gamma}^i)_{i\in I}\in\mathcal
E^+(\boldsymbol{A})$ {\it extremal\/} in the
$(\boldsymbol{A},\boldsymbol{a},g,\chi,J)$-problem if there exists
$(\boldsymbol\mu_s)_{s\in S}\in\mathbb
M_\chi(\boldsymbol{A},\boldsymbol{a},g,J)$ that converges
to~$\tilde{\boldsymbol\gamma}$ both strongly and vaguely; such a net
$(\boldsymbol\mu_s)_{s\in S}$ is said to {\it
generate\/}~$\tilde{\boldsymbol\gamma}$. The class of all
those~$\tilde{\boldsymbol\gamma}$ will be denoted by $\mathfrak
E_\chi(\boldsymbol{A},\boldsymbol{a},g,J)$.\footnote{In the case
$J=I$, the tilde in the notation of an extremal measure will often
be dropped.}
\end{definition}

It follows from Lemma~\ref{lemma:lower} with $\psi=g$ that, for any
$\tilde{\boldsymbol\gamma}\in\mathfrak
E_\chi(\boldsymbol{A},\boldsymbol{a},g,J)$,
\begin{equation}\label{inequal}\langle
g,\tilde{\gamma}^j\rangle\leqslant a_j\quad\text{for all \ }j\in
J.\end{equation} Also observe that
\begin{equation}\label{incl}\mathfrak S_\chi(\boldsymbol{A},\boldsymbol{a},g,J)\subset
\mathfrak E_{\chi}(\boldsymbol{A},\boldsymbol{a},g,J),\end{equation}
since, because of Corollary~\ref{min:proj}, each
$\tilde{\boldsymbol\lambda}\in\mathfrak
S_\chi(\boldsymbol{A},\boldsymbol{a},g,J)$ (provided exists) can be
treated as an extremal measure generated by the constant sequence
$(\tilde{\boldsymbol\lambda})_{n\in\mathbb N}\in\mathbb
M_\chi(\boldsymbol{A},\boldsymbol{a},g,J)$.

\begin{lemma}\label{lemma:WM} Furthermore, the following assertions hold true:
\begin{itemize}
\item[\rm(i)] From every minimizing net one can select a subnet generating an extremal measure;
hence, $\mathfrak E_{\chi}(\boldsymbol{A},\boldsymbol{a},g,J)$ is
nonempty. Moreover,
\begin{equation}
\mathfrak E_{\chi}(\boldsymbol{A},\boldsymbol{a},g,J)=\mathcal
M_{\chi}(\boldsymbol{A},\boldsymbol{a},g,J).
\label{WM}\end{equation}
\item[\rm(ii)]  For any\/ $\tilde{\boldsymbol\gamma}_1,\tilde{\boldsymbol\gamma}_2\in\mathfrak
E_{\chi}(\boldsymbol{A},\boldsymbol{a},g,J)$, it holds that\/
$\|\tilde{\boldsymbol\gamma}_1-\tilde{\boldsymbol\gamma}_2\|_{\mathcal
E^+(\boldsymbol A)}=0$. Hence,
$R\tilde{\boldsymbol{\gamma}}_1=R\tilde{\boldsymbol{\gamma}}_2$,
while\/
$\tilde{\boldsymbol{\gamma}}_1=\tilde{\boldsymbol{\gamma}}_2$ if and
only if all\/ $A_i$, $i\in I$, are mutually disjoint.\smallskip
\item[\rm(iii)] $\mathfrak E_{\chi}(\boldsymbol{A},\boldsymbol{a},g,J)$ is vaguely compact.
\end{itemize}
\end{lemma}

\begin{proof}Choose arbitrary $(\boldsymbol\mu_s)_{s\in S}$ and $(\boldsymbol\xi_t)_{t\in T}$ in $\mathbb
M_{\chi}(\boldsymbol{A},\boldsymbol{a},g,J)$. Then
\begin{equation}
\lim_{(s,\,t)\in S\times
T}\,\|\boldsymbol\mu_s-\boldsymbol\xi_t\|_{\mathcal E^+(\boldsymbol
A)}=0, \label{fund}
\end{equation}
where $S\times T$ is the directed product of the directed sets~$S$
and~$T$ (see, e.g.,~\cite[Chapter~2, Section~3]{K}). Since $\mathcal
E_H^+(\boldsymbol{A},\boldsymbol{a},g,J)$ is convex (cf.~the
footnote to Corollary~\ref{lemma3.3}), we get
\[4G_{\chi}(\boldsymbol{A},\boldsymbol{a},g,J)\leqslant4G_{\chi}\Biggl(\frac{\boldsymbol\mu_s+\boldsymbol\xi_t}{2}\Biggr)=
\|R\boldsymbol\mu_s+R\xi_t\|^2+4\kappa(\chi,R\boldsymbol\mu_s+R\boldsymbol\xi_t).\]
On the other hand, applying the parallelogram identity in the
pre-Hilbert space~$\mathcal E$ to $R\boldsymbol\mu_s$ and~$R\xi_t$
and then adding and subtracting
$4\kappa(\chi,R\boldsymbol\mu_s+R\boldsymbol\xi_t)$, we obtain
\[\|R\boldsymbol\mu_s-R\boldsymbol\xi_t\|^2= -\|R\boldsymbol\mu_s+R\boldsymbol\xi_t\|^2-4\kappa(\chi,R\boldsymbol\mu_s+R\boldsymbol\xi_t)+
2G_{\chi}(\boldsymbol\mu_s)+2G_{\chi}(\boldsymbol\xi_t).\] When
combined with the preceding relation, this gives
\[0\leqslant\|R\boldsymbol\mu_s-R\boldsymbol\xi_t\|^2\leqslant-4G_{\chi}(\boldsymbol{A},\boldsymbol{a},g,J)+
2G_{\chi}(\boldsymbol\mu_s)+2G_{\chi}(\boldsymbol\xi_t),\] which
yields (\ref{fund}) when combined with (\ref{min}).

Relation~(\ref{fund}) implies that $(\boldsymbol\mu_s)_{s\in S}$ is
strongly fundamental. By~Theorem~\ref{th:complete} with $b=H$, for
every vague cluster point~$\boldsymbol\mu$
of~$(\boldsymbol\mu_s)_{s\in S}$ (it exists) we therefore get
$\boldsymbol\mu_s\to\boldsymbol\mu$ strongly. This shows that
$\boldsymbol\mu$ is an extremal measure; actually, $\mathcal
M_{\chi}(\boldsymbol{A},\boldsymbol{a},g,J)\subset\mathfrak
E_{\chi}(\boldsymbol{A},\boldsymbol{a},g,J)$. Since the converse
inclusion is obvious, relation~(\ref{WM}) follows.

Having thus proved assertion~(i), we proceed by verifying (ii).
Choose $(\boldsymbol\mu_s)_{s\in S}$ and $(\boldsymbol\xi_t)_{t\in
T}$ generating $\tilde{\boldsymbol\gamma}_1$ and
$\tilde{\boldsymbol\gamma}_2$, respectively. Then application
of~(\ref{fund}) shows that $(\boldsymbol\mu_s)_{s\in S}$ converges
to both~$\tilde{\boldsymbol\gamma}_1$
and~$\tilde{\boldsymbol\gamma}_2$ strongly, hence
$\|\tilde{\boldsymbol\gamma}_1-\tilde{\boldsymbol\gamma}_2\|_{\mathcal
E^+(\boldsymbol A)}=0$, and consequently
$\|R\tilde{\boldsymbol\gamma}_1-R\tilde{\boldsymbol\gamma}_2\|=0$
by~(\ref{twonorms}). In view of the strict positive definiteness of
the kernel, assertion~(ii) follows.

To establish assertion~(iii), fix
$(\tilde{\boldsymbol\gamma}_s)_{s\in S}\subset\mathfrak
E_{\chi}(\boldsymbol{A},\boldsymbol{a},g,J)$; then
$(\tilde{\boldsymbol\gamma}_s)_{s\in S}\subset\mathcal
E^+_H(\boldsymbol A)$. Since the class $\mathfrak M^+_H(\boldsymbol
A)$ is vaguely compact, there exists a vague cluster
point~$\tilde{\boldsymbol\gamma}_0$ of
$(\tilde{\boldsymbol\gamma}_s)_{s\in S}$. Let
$(\tilde{\boldsymbol\gamma}_t)_{t\in T}$ be a subnet
of~$(\tilde{\boldsymbol\gamma}_s)_{s\in S}$ that converges vaguely
to~$\tilde{\boldsymbol\gamma}_0$; then for every $t\in T$ one can
choose a net $(\boldsymbol\mu_{s_t})_{s_t\in S_t}\in\mathbb
M_{\chi}(\boldsymbol{A},\boldsymbol{a},g,J)$ converging vaguely
to~$\tilde{\boldsymbol\gamma}_t$. Consider the Cartesian product
$\prod_{t\in T}S_t$~--- that is, the collection of all
functions~$\beta$ on~$T$ with $\beta(t)\in S_t$, and let~$D$ denote
the directed product $T\times\prod_{t\in T}S_t$. For every
$(t,\beta)\in D$, write
$\boldsymbol\mu_{(t,\beta)}:=\boldsymbol\mu_{\beta(t)}$. Then the
theorem on iterated limits from \cite[Chapter~2, Section~4]{K}
implies that the net $\boldsymbol\mu_{(t,\beta)}$, $(t,\beta)\in D$,
belongs to $\mathbb M_{\chi}(\boldsymbol{A},\boldsymbol{a},g,J)$ and
converges vaguely to~$\tilde{\boldsymbol\gamma}_0$; thus,
$\tilde{\boldsymbol\gamma}_0\in\mathcal
M_{\chi}(\boldsymbol{A},\boldsymbol{a},g,J)$. Combined
with~(\ref{WM}), this yields
assertion~(iii).
\end{proof}

\begin{corollary}\label{cor:uniqu}Any two minimizers\/
$\tilde{{\boldsymbol\lambda}}_1,\tilde{{\boldsymbol\lambda}}_2\in\mathfrak
S_{\chi}(\boldsymbol{A},\boldsymbol{a},g,J)$ {\rm(}provided
exist\/{\rm)} are equivalent in\/~$\mathcal E^+(\boldsymbol A)$.
Hence,
$R\tilde{\boldsymbol{\lambda}}_1=R\tilde{\boldsymbol{\lambda}}_2$,
while\/
$\tilde{\boldsymbol{\lambda}}_1=\tilde{\boldsymbol{\lambda}}_2$ if
and only if all\/ $A_i$, $i\in I$, are mutually
disjoint.\end{corollary}

\begin{proof}This is a direct consequence of inclusion~(\ref{incl})
and assertion~(ii) of Lemma~\ref{lemma:WM}.
\end{proof}

\begin{corollary}\label{IorII} For every extremal measure\/ ${\tilde{\boldsymbol\gamma}}\in\mathfrak
E_{\chi}(\boldsymbol{A},\boldsymbol{a},g,J)$ it holds that
\begin{equation}\label{eqIorII}
G_{\chi}({\tilde{\boldsymbol\gamma}})=G_{\chi}(\boldsymbol{A},\boldsymbol{a},g,J).
\end{equation}
\end{corollary}

\begin{proof}  Applying (\ref{gmu}) to $\tilde{\boldsymbol{\gamma}}\in\mathfrak
E_{\chi}(\boldsymbol{A},\boldsymbol{a},g,J)$ and each
of~$\boldsymbol{\mu}_s$, $s\in S$, where $(\boldsymbol{\mu}_s)_{s\in
S}$ is a net generating~$\tilde{\boldsymbol{\gamma}}$, and using the
fact that $\boldsymbol{\mu}_s\to\tilde{\boldsymbol{\gamma}}$
strongly, we get
\[G_{\chi}(\tilde{\boldsymbol\gamma})=\|\chi+R\tilde{\boldsymbol\gamma}\|^2-\|\chi\|^2=\lim_{s\in S}\,
\Bigl[\|\chi+R\boldsymbol\mu_s\|^2-\|\chi\|^2\Bigr]=\lim_{s\in
S}\,G_{\chi}(\boldsymbol\mu_s),\] which yields~(\ref{eqIorII}) when
combined with~(\ref{min}).
\end{proof}

\begin{corollary}\label{extr.min}Let\/ $J=I$. If the\/ $(\boldsymbol{A},\boldsymbol{a},g,\chi)$-problem is
solvable, then the class\/ $\mathfrak
S_\chi(\boldsymbol{A},\boldsymbol{a},g)$ of all its solutions is
compact in the vague topology; moreover,
\begin{equation}\label{hrryu}\mathfrak
S_\chi(\boldsymbol{A},\boldsymbol{a},g)=\mathfrak
E_\chi(\boldsymbol{A},\boldsymbol{a},g).\end{equation}
\end{corollary}

\begin{proof}Due to
assertion~(iii) of Lemma~\ref{lemma:WM}, the corollary will follow
once we prove~(\ref{hrryu}). As is seen from relations~(\ref{incl})
and~(\ref{eqIorII}), to this end it is enough to show that, for any
given $\boldsymbol\gamma\in\mathfrak
E_\chi(\boldsymbol{A},\boldsymbol{a},g)$,
\begin{equation}\label{ex.min}\langle g,\gamma^i\rangle=a_i\quad\text{for all \
}i\in I.\end{equation} Because of assertion~(ii) of
Lemma~\ref{lemma:WM}, for any $\boldsymbol\lambda\in\mathfrak
S_\chi(\boldsymbol{A},\boldsymbol{a},g)$ one has
$R\boldsymbol\gamma=R\boldsymbol\lambda$, so
$|R\boldsymbol\gamma|=|R\boldsymbol\lambda|$, hence $\langle
g,|R\boldsymbol\gamma|\rangle=\langle
g,|R\boldsymbol\lambda|\rangle$ and, as an application of
Lemma~\ref{lemma3.5} with $\psi=g$,
\[\sum_{i\in I}\,\langle g,\gamma^i\rangle=\sum_{i\in I}\,\langle
g,\lambda^i\rangle=|\boldsymbol a|,\] $|\boldsymbol a|$ being finite
by assumption~(\ref{ag}). In view of relation~(\ref{inequal}), this
yields~(\ref{ex.min}).
\end{proof}

\section{Proof of Theorem~\ref{lemma:exist}}\label{proof:lemma:exist}

Given $J$, where $I^+\subseteq J\subseteq I$, fix an extremal
measure $\tilde{\boldsymbol\gamma}\in\mathfrak
E_{\chi}(\boldsymbol{A},\boldsymbol{a},g,J)$; it exists according to
assertion~(i) of Lemma~\ref{lemma:WM}. A part of the proof of
Theorem~\ref{lemma:exist} is given as the following lemma, to be
needed also in the sequel.

\begin{lemma}\label{lemma:lemma}If\/ $C(A_j)<\infty$ for some\/ $j\in J$, then
\begin{equation}
\langle g,\tilde{\gamma}^j\rangle=a_j.\label{24}
\end{equation}
\end{lemma}

\begin{proof} We treat $A_j$ as a locally
compact space with the topology induced from~$\mathrm X$. Given a
set $E\subset A_j$, let $\chi_E$ denote its characteristic function
and $\theta_E\in\mathcal E^+(\overline{E})$ the interior equilibrium
measure associated with~$E$ (as the kernel is perfect, the existence
of~$\theta_E$ is guaranteed by Theorem~\ref{thF:1}). Also write
$E^c:=E^{c_{A_j}}:=A_j\setminus E$, and let $\{K_j\}_{A_j}$ be the
increasing filtering family of all compact subsets~$K_j$ of~$A_j$.

We then observe that there is no loss of generality in assuming
$g|_{A_j}$ to satisfy condition~(\ref{growth}) with some
$r_j\in(1,\infty)$ and $\nu_j\in\mathcal E$. Indeed, otherwise
$g|_{A_j}$ has to be bounded from above (say by~$M$), which combined
with relation~(\ref{equ:F2}) for $E=A_j$ results in
inequality~(\ref{growth}) with $\nu_j:=M^{r_j}\theta_{A_j}$,
$r_j\in(1,\infty)$ being arbitrary.

To establish (\ref{24}),  choose a (strongly bounded) net
$(\boldsymbol\mu_s)_{s\in S}\in\mathbb
M_{\chi}(\boldsymbol{A},\boldsymbol{a},g,J)$
generating~$\tilde{\boldsymbol\gamma}$. Since $g\chi_{K_j}$ is upper
semicontinuous on~$A_j$ while $\mu_s^j\to\tilde{\gamma}^j$ vaguely,
we get
\[
\langle g\chi_{K_j},\tilde{\gamma}^j\rangle\geqslant\limsup_{s\in
S}\,\langle g\chi_{K_j},\mu_s^j\rangle\quad\mbox{for every \ }
K_j\in\{K_j\}_{A_j}.\] On the other hand, application of Lemma~1.2.2
from~\cite{F1} yields
\[
\langle g,\tilde{\gamma}^j\rangle=\lim_{K_j\uparrow A_j}\,\langle
g\chi_{K_j},\tilde{\gamma}^j\rangle,\] which together with the
preceding inequality and relation~(\ref{inequal}) gives
\[
a_j\geqslant\langle
g,\tilde{\gamma}^j\rangle\geqslant\limsup_{(s,\,K_j)\in
S\times\{K_j\}_{A_j}}\,\langle g\chi_{K_j},\mu_s^j\rangle=
a_j-\liminf_{(s,\,K_j)\in S\times\{K_j\}_{A_j}}\,\langle
g\chi_{K^c_j},\mu_s^j\rangle,\] $S\times\{K_j\}_{A_j}$ being the
directed product of the directed sets~$S$ and~$\{K_j\}_{A_j}$.
Hence, if we prove
\begin{equation}
\liminf_{(s,\,K_j)\in S\times\{K_j\}_{A_j}}\,\langle
g\chi_{K^c_j},\mu_s^j\rangle=0, \label{25}
\end{equation}
the lemma follows.

For any $K_j,\widehat{K}_j\in\{K_j\}_{A_j}$ such that
$K_j\subset\widehat{K}_j$, consider $\theta_{K^c_j}$
and~$\theta_{\widehat{K}^c_j}$, the interior equilibrium measures
associated with~$K^c_j$ and~$\widehat{K}^c_j$, respectively. Then
$\theta_{\widehat{K}^c_j}$ minimizes the energy in the
class~$\Gamma_{\widehat{K}^c_j}$, while $\theta_{K^c_j}$ belongs to
$\Gamma_{\widehat{K}^c_j}$; see Theorem~\ref{thF:1}. By
Lemma~\ref{lemma:convex} with $\mathcal H=\Gamma_{\widehat{K}^c_j}$
and $\nu=\theta_{K^c_j}$, this yields
\[
\|\theta_{K^c_j}-\theta_{\widehat{K}^c_j}\|^2\leqslant
\|\theta_{K^c_j}\|^2-\|\theta_{\widehat{K}^c_j}\|^2.\] But, as is
seen from~(\ref{equF:1}), the net $\|\theta_{K^c_j}\|$,
$K_j\in\{K_j\}_{A_j}$, is bounded and nonincreasing, and so it is
fundamental in~$\mathbb R$. The preceding inequality then yields
that the net $(\theta_{K^c_j})_{K_j\in\{K_j\}_{A_j}}$ is strongly
fundamental in~$\mathcal E^+$. Since, clearly, it converges vaguely
to zero, the property~(C$_1$) (see~Section~\ref{sec:2}) implies that
zero is also its strong limit; hence,
\begin{equation}
\lim_{K_j\uparrow A_j}\,\|\theta_{K^c_j}\|=0. \label{27}
\end{equation}

Write $q_j:=r_j(r_j-1)^{-1}$, where $r_j\in(1,\infty)$ is the number
involved in condition~(\ref{growth}). Combining
relations~(\ref{growth}) and~(\ref{equ:F2}), the latter with
$E=K^c_j$, shows that the inequality
\[
g(x)\chi_{K^c_j}(x)\leqslant\kappa(x,\nu_j)^{1/r_j}
\kappa(x,\theta_{K^c_j})^{1/q_j}\]  subsists n.e.~in~$A_j$, and
hence $\mu_s^j$-a.e.~in~$\mathrm X$ by Lemma~\ref{a.e.}. Having integrated
this relation with respect to~$\mu_s^j$, we then apply the H\"older
and, subsequently, the Cauchy--Schwarz inequalities to the integrals
on the right. This gives
\[\langle
g\chi_{K^c_j},\mu_s^j\rangle\leqslant\Biggl[\int\kappa(x,\nu_j)\,d\mu_s^j(x)\Biggr]^{1/r_j}\,
\Biggl[\int\kappa(x,\theta_{K^c_j})\,d\mu_s^j(x)\Biggr]^{1/q_j}\leqslant
\|\nu_j\|^{1/r_j}\|\theta_{K^c_j}\|^{1/q_j}\|\mu_s^j\|.\] Taking
limits here along $S\times\{K_j\}_{A_j}$ and using
relations~(\ref{twostars}) and (\ref{27}), we get~(\ref{25}) as
desired.
\end{proof}

Let now assumption~(\ref{capfinite}) in Theorem~\ref{lemma:exist}
hold. Then, in consequence of Lemma~\ref{lemma:lemma}, the extremal
measure $\tilde{\boldsymbol\gamma}$ belongs to~$\mathcal
E^+(\boldsymbol A,\boldsymbol a,g,J)$. Because of~(\ref{eqIorII}),
this yields that, actually, $\tilde{\boldsymbol\gamma}$ is a
solution to the $(\boldsymbol A,\boldsymbol a, g,\chi,J)$-problem,
and the statement on the solvability follows.

It has thus been proved that $\mathfrak E_\chi(\boldsymbol
A,\boldsymbol a,g,J)\subset\mathfrak S_\chi(\boldsymbol
A,\boldsymbol a,g,J)$. Since the converse also holds according to
inclusion~(\ref{incl}), these two classes have to be equal. On
account of assertion~(iii) of Lemma~\ref{lemma:WM}, this implies the
vague compactness of
$\mathfrak S_\chi(\boldsymbol A,\boldsymbol a,g,J)$.\hfill$\square$

\section{Proof of
Theorem~\ref{th:solvable1}}\label{proof:th:solvable1}

Recall that $L\subseteq I^-$ consists of $\ell\in I$ with
$C(A_\ell)=\infty$. Fix $\tilde{\boldsymbol\lambda}\in\mathfrak
S_\chi(\boldsymbol A,\boldsymbol a,g,CL)$ (it exists due to
Theorem~\ref{lemma:exist} with $J=CL$) and assume $a_\ell$, $\ell\in
L$, to satisfy relation~(\ref{br}).

Observe that $\langle g,\tilde{\lambda}^i\rangle$ is finite for each
$i\in I$, so that inequalities~(\ref{br}) make sense. Indeed, since
$\tilde{\lambda}^i(\mathrm X)\leqslant H<\infty$ according to
Corollary~\ref{min:proj}, one can assume $g$ to satisfy
condition~(\ref{growth}), for if not, then $g$ has to be bounded
from above and the finiteness of $\langle
g,\tilde{\lambda}^i\rangle$ is obvious. By the H\"older and
Cauchy--Schwarz inequalities, we then obtain
\[\langle g,\tilde{\lambda}^i\rangle\leqslant
\Biggl[\int\kappa(x,\nu_i)\,d\tilde{\lambda}^i(x)\Biggr]^{1/r_i}\,
\Biggl[\int1\,d\tilde{\lambda}^i\Biggr]^{1/q_i}\leqslant
\|\nu_i\|^{1/r_i}\|\tilde{\lambda}^i\|^{1/r_i} H^{1/q_i}<\infty\] as
desired. Here, $r_i$ and $\nu_i$ are the same as in
condition~(\ref{growth}) and $q_i:=r_i(r_i-1)^{-1}$.

To establish the theorem, for each $\ell\in L$ we choose probability
measures $\tau^\ell_n\in\mathcal E^+(A_\ell)$, $n\in\mathbb N$, such
that $\tau^\ell_n\to0$ vaguely as $n\to\infty$ and
\begin{equation}\label{tau}\|\tau^\ell_n\|\leqslant n^{-1}.\end{equation}
Such $\tau^\ell_n$ exist due to the fact that $C(A_\ell\setminus
K)=\infty$ for any compact~$K$, which in turn follows from
$C(A_\ell)=\infty$ because of the strict positive definiteness of
the kernel.

Further, for each $n\in\mathbb N$ define
${\hat{\boldsymbol\tau}}_n=({\hat{\tau}}^i_n)_{i\in I}$, where
$\hat{\tau}^i_n:=0$ for all $i\in CL$ and
\begin{equation*}\label{hattau}\hat{\tau}^i_n:=\frac{[a_i-\langle
g,\tilde{\lambda}^i\rangle]{\tau}^i_n}{\langle
g,\tau^i_n\rangle}\quad\text{otherwise}.\end{equation*} Since
$\langle g,\tau^\ell_n\rangle\geqslant g_{\inf}>0$ for all $\ell\in
L$ because of assumption~(\ref{ginf}), we have
\begin{equation}\label{tauvague}{\hat{\boldsymbol\tau}}_n\to\boldsymbol{0}\quad\text{vaguely
(as \ $n\to\infty$)}\end{equation} and also, due to relations
(\ref{ag}) and~(\ref{tau}),
\begin{equation}\label{on}\sum_{i\in
I}\,\|\hat{\tau}^i_n\|\leqslant |\boldsymbol
a_L|g_{\inf}^{-1}n^{-1}<\infty,\quad n\in\mathbb N.\end{equation}

Estimate~(\ref{on}) yields that
${\hat{\boldsymbol{\tau}}}_n\in\mathcal E^+(\boldsymbol A)$ for all
$n\in\mathbb N$ (see~Section~\ref{sec:cond}); hence, in view of the
convexity of~$\mathcal E^+(\boldsymbol A)$,
\begin{equation}\label{raz3}\boldsymbol\mu_n:=\tilde{\boldsymbol\lambda}+{\hat{\boldsymbol{\tau}}}_n\in\mathcal
E^+(\boldsymbol A,\boldsymbol a,g),\quad n\in\mathbb
N,\end{equation}  and consequently, by (\ref{gmu}),
\begin{align}\notag G_\chi(\boldsymbol A,\boldsymbol a,\boldsymbol g)\leqslant
G_\chi(\boldsymbol\mu_n)&=-\|\chi\|^2+\|\chi+R\boldsymbol\mu_n\|^2\\{}&
=-\|\chi\|^2+\|\chi+R\tilde{\boldsymbol\lambda}+R{\hat{\boldsymbol{\tau}}}_n\|^2\leqslant
G_\chi(\tilde{\boldsymbol\lambda})+c(n),\label{raz1}\end{align}
where
$c(n):=\|R{\hat{\boldsymbol{\tau}}}_n\|\bigl(\|R{\hat{\boldsymbol{\tau}}}_n\|+2\|\chi+R\tilde{\boldsymbol\lambda}\|\bigr)$.
But
\begin{equation}\label{raz2}G_\chi(\tilde{\boldsymbol\lambda})=G_\chi(\boldsymbol
A,\boldsymbol a,g,CL)\leqslant G_\chi(\boldsymbol A,\boldsymbol
a,g),\end{equation} while $c(n)\to0$ as $n\to\infty$, the latter
being a consequence of estimate~(\ref{on}). Combining
relations~(\ref{raz1}) and~(\ref{raz2}) and then letting
$n\to\infty$, we thus get
\begin{equation*}\label{raz4}G_\chi(\tilde{\boldsymbol\lambda})=G_\chi(\boldsymbol A,\boldsymbol
a,g,CL)=G_\chi(\boldsymbol A,\boldsymbol a,g)=
\lim_{n\to\infty}\,G_\chi(\boldsymbol\mu_n).\end{equation*} This
establishes assertion~(\ref{LLL}) and, on account of
inclusion~(\ref{raz3}), also the fact that
\[(\boldsymbol\mu_n)_{n\in\mathbb N}\in\mathbb M_\chi(\boldsymbol A,\boldsymbol a,g).\]

As $\boldsymbol\mu_n\to\tilde{\boldsymbol\lambda}$ vaguely because
of relation~(\ref{tauvague}), the last inclusion yields
$\tilde{\boldsymbol\lambda}\in\mathcal M_\chi(\boldsymbol
A,\boldsymbol a,g)$ and hence, by~(\ref{WM}) with $J=I$,
\begin{equation*}\label{hrr}\tilde{\boldsymbol\lambda}\in\mathfrak E_\chi(\boldsymbol A,\boldsymbol
a,g).\end{equation*} Using Corollary~\ref{extr.min}, we then see
that the main $(\boldsymbol A,\boldsymbol a,g,\chi)$-problem is
solvable if and only if $\tilde{\boldsymbol\lambda}$ serves as one
of its minimizers, which in turn, due to~(\ref{eqIorII}) with $J=I$,
holds if and only if
\[\tilde{\boldsymbol\lambda}\in\mathcal E^+(\boldsymbol A,\boldsymbol a,g).\] Since the latter is equivalent
to the requirement that all the inequalities~(\ref{br}) are
equalities, the proof is complete.\hfill$\square$

\section{Description of the $\boldsymbol f$-weighted extremal
potentials in the Gauss variational problem}\label{sec:descr}

It is seen from (\ref{Rpot}) that, for any
$\boldsymbol{\mu}\in\mathcal E^+(\boldsymbol{A})$, the $\boldsymbol
f$-weighted vector potential $\boldsymbol
W_{\boldsymbol\mu}=(W_{\boldsymbol\mu}^i)_{i\in I}$ is finite
n.e.~in~$\mathrm X$ and its components can be written in the form
\begin{equation}\label{wpopot}W_{\boldsymbol\mu}^i(x)=\alpha_i\kappa(x,\chi+R\boldsymbol\mu)\quad\text{n.e.~in
\ }\mathrm X.\end{equation} Therefore, for any
$\boldsymbol{\mu},\boldsymbol{\mu}_1\in\mathcal
E^+(\boldsymbol{A})$,
\begin{equation}\label{xxx}\bigl\langle\boldsymbol
W_{\boldsymbol\mu},\boldsymbol{\mu}_1\bigr\rangle=\sum_{i\in
I}\,\bigl\langle
W^i_{\boldsymbol\mu},\mu_1^i\bigr\rangle=\kappa(\chi+R\boldsymbol{\mu},R\boldsymbol{\mu}_1),\end{equation}
which follows from Lemma~\ref{lemma3.5} when applied to
$R\boldsymbol\mu_1$ and each of the functions
$\kappa(\cdot,\chi^++R\boldsymbol\mu^+)$ and
$\kappa(\cdot,\chi^-+R\boldsymbol\mu^-)$.

Let $\boldsymbol\gamma$ be extremal in the Gauss variational
problem. Taking~(\ref{gmu}) and~(\ref{eqIorII}) into account, we
then conclude from~(\ref{xxx}) with
$\boldsymbol\mu=\boldsymbol\mu_1=\boldsymbol\gamma$ that
\begin{equation}\label{wsum}\sum_{i\in I}\,\bigl\langle
W^i_{\boldsymbol\gamma},\gamma^i\bigr\rangle=\frac12\Bigl[\|\boldsymbol\gamma\|^2_{\mathcal
E^+(\boldsymbol A)}+G_{\chi}(\boldsymbol A,\boldsymbol
a,g)\Bigr],\end{equation} where the expression on the right (hence,
also that on the left) does not depend on the choice
of~$\boldsymbol\gamma$ in consequence of assertion~(ii) of
Lemma~\ref{lemma:WM}.\footnote{Cf.~also Corollary~\ref{eta3} below.}

The proof of Theorem~\ref{th:solvable2}, to be given in
Section~\ref{proof:th:solvable2}, is based on a description of the
$\boldsymbol f$-weighted extremal potential $\boldsymbol
W_{\boldsymbol\gamma}=(W_{\boldsymbol\gamma}^i)_{i\in I}$ provided
as follows.

\begin{theorem}\label{th:descpot1} Given\/ $\boldsymbol\gamma\in\mathfrak
E_\chi(\boldsymbol{A},\boldsymbol{a},g)$, it holds
that\/\footnote{Recall that $C(A_i)>0$ for all $i\in I$ (see the
footnote to assumption~(\ref{G:finite})).}
\begin{equation}\label{descpot1}
a_iW_{\boldsymbol\gamma}^i(x)\geqslant\bigl\langle
W_{\boldsymbol\gamma}^i,\gamma^i\bigr\rangle g(x)\quad\mbox{n.e.~in
\ } A_i,\quad i\in I.\end{equation} Relations\/~{\rm(\ref{wsum})}
and\/~{\rm(\ref{descpot1})} characterize uniquely the vector\/
$\langle W_{\boldsymbol\gamma}^i,\gamma^i\rangle$, $i\in I$, in the
sense that, if these two relations are satisfied by some\/~$\eta_i$
in place of\/~$\langle W_{\boldsymbol\gamma}^i,\gamma^i\rangle$,
then
\begin{equation}\label{etaeta}\eta_i=\bigl\langle
W_{\boldsymbol\gamma}^i,\gamma^i\bigr\rangle\quad\text{for all \
}i\in I.\end{equation}
\end{theorem}

\begin{proof}Having fixed $i\in I$, we start with the observation that, for any $\xi\in\mathcal
E$, $\langle W_{\boldsymbol\mu}^i,\xi\rangle$ is a strongly
continuous function of $\boldsymbol\mu\in\mathcal E^+(\boldsymbol
A)$. Indeed, this is clear from representation~(\ref{wpopot}) in
view of the isometry between $\mathcal E^+(\boldsymbol A)$
and~$R\mathcal E^+(\boldsymbol A)$.

Choose a net $(\boldsymbol\mu_s)_{s\in S}\in\mathbb
M_\chi(\boldsymbol{A},\boldsymbol{a},g)$ generating the extremal
measure~$\boldsymbol\gamma$. Then the above observation can slightly
be generalized in the following way:
\begin{equation}\label{threestars}\bigl\langle
W_{\boldsymbol\gamma}^i,\gamma^i\bigr\rangle=\lim_{s\in
S}\,\bigl\langle
W_{{\boldsymbol\mu}_s}^i,\mu_s^i\bigr\rangle.\end{equation}

Indeed, according to relation~(\ref{twostars}), the net
$(\mu_s^i)_{s\in S}$ is strongly bounded (say by~$M$). Since it
converges to~$\gamma^i$ vaguely, the property~$(C_2)$ (see
Section~\ref{sec:2}) yields that $\mu_s^i\to\gamma^i$ weakly; hence,
for every $\varepsilon>0$,
\[|\kappa(\chi+R\boldsymbol\gamma,\mu^i_s-\gamma^i)|<\varepsilon\]
whenever $s\in S$ is large enough. Furthermore, by the
Cauchy--Schwarz inequality,
\[|\kappa(\chi+R\boldsymbol\gamma,\mu^i_s)-\kappa(\chi+R\boldsymbol\gamma_s,\mu^i_s)|\leqslant
\|\mu^i_s\|\,\|R\boldsymbol\gamma-R\boldsymbol\mu_s\|\leqslant
M\|R\boldsymbol\gamma-R\boldsymbol\mu_s\|.\] As
$R\boldsymbol\mu_s\to R\boldsymbol\gamma$ strongly, the last two
relations combined give
\[\kappa(\chi+R\boldsymbol\gamma,\gamma^i)=\lim_{s\in
S}\,\kappa(\chi+R\boldsymbol\mu_s,\mu^i_s),\] which is equivalent
to~(\ref{threestars}) because of representation~(\ref{wpopot}).

In order to establish assertion~(\ref{descpot1}), we now assume, on
the contrary, that there is a set $E\subset A_i$ of interior
capacity nonzero with the property
\begin{equation*}\label{heart'}
a_iW_{\boldsymbol\gamma}^i(x)<\bigl\langle
W_{\boldsymbol\gamma}^i,\gamma^i\bigr\rangle g(x)\quad\mbox{for all
\ } x\in E.
\end{equation*}
Having chosen $\omega\in\mathcal E^+_E$ so that $\langle
g,\omega\rangle=a_i$ (such a measure exists), we then obtain
\begin{equation}\label{threeheart}
\bigl\langle W_{\boldsymbol\gamma}^i,\omega\bigr\rangle<\bigl\langle
W_{\boldsymbol\gamma}^i,\gamma^i\bigr\rangle.
\end{equation}

To get a contradiction, for all $r\in(0,1)$ and $s\in S$ we define
$\hat{\boldsymbol{\mu}}_s=(\hat{\mu}_s^k)_{k\in I}$, where
\[\hat{\mu}_s^k:=\left\{
\begin{array}{cl} {\mu}_s^i-r({\mu}_s^i-\omega)\quad & \mbox{if \ } k=i,\\ {\mu}_s^k & \mbox{otherwise}.\\
\end{array} \right.\]
Since then
\begin{equation*}\label{rhat}R\hat{\boldsymbol{\mu}}_s=-\alpha_ir({\mu}_s^i-\omega)+R\boldsymbol{\mu}_s,\end{equation*}
Corollary~\ref{lemma3.3} implies
$\hat{\boldsymbol{\mu}}_s\in\mathcal E^+(\boldsymbol A)$. Actually,
$\hat{\boldsymbol{\mu}}_s\in\mathcal E^+(\boldsymbol A,\boldsymbol
a,g)$ for all $s\in S$, which in view of relations~(\ref{gmu}),
(\ref{twostars}), and~(\ref{wpopot}) yields
\begin{align*}G_\chi(\boldsymbol A,\boldsymbol
a,g)&\leqslant
G_\chi(\hat{\boldsymbol{\mu}}_s)=\|R\hat{\boldsymbol{\mu}}_s\|^2+2\kappa(\chi,R\hat{\boldsymbol{\mu}}_s)\\{}&=
\|R{\boldsymbol{\mu}}_s\|^2-2\alpha_ir\kappa(R{\boldsymbol{\mu}}_s,{\mu}_s^i-\omega)+r^2\|{\mu}_s^i-\omega\|^2+
2\kappa(\chi,R{\boldsymbol{\mu}}_s)-2\alpha_ir\kappa(\chi,{\mu}_s^i-\omega)\\{}&\leqslant
G_\chi({\boldsymbol{\mu}}_s)- 2r\bigl\langle
W^i_{\boldsymbol{\mu}_s},{\mu}_s^i-\omega\bigr\rangle +r^2M_1,
\end{align*}
$M_1$ being positive. Passing here to the limit as $s$ ranges along
$S$, on account of~(\ref{min}), (\ref{threestars}) and the
continuity of $\bigl\langle
W^i_{\boldsymbol{\mu}},\omega\bigr\rangle$ relative to
$\boldsymbol\mu\in\mathcal E^+(\boldsymbol A)$ we get
\[0\leqslant-2r\bigl\langle W^i_{\boldsymbol\gamma},\gamma^i-\omega\bigr\rangle+r^2M_1,\]
which leads to $\langle
W^i_{\boldsymbol\gamma},\gamma^i-\omega\rangle\leqslant0$ by letting
$r\to0$. This contradicts to inequality~(\ref{threeheart}).

To complete the proof, assume now relations~(\ref{wsum})
and~(\ref{descpot1}) to hold also with some~$\eta_i$, $i\in I$, in
place of~$\langle W_{\boldsymbol\gamma}^i,\gamma^i\rangle$. Using
the fact that the union of two universally measurable sets with
interior capacity zero has interior capacity zero as well
(see~\cite{F1}), for each $i\in I$ we then have
\begin{equation*}\label{opana}
a_iW_{\boldsymbol\gamma}^i(x)\geqslant\max\Bigl\{\eta_i,\bigl\langle
W_{\boldsymbol\gamma}^i,\gamma^i\bigr\rangle\Bigr\}
g(x)\quad\mbox{n.e. in \ }A_i,\end{equation*} which according to
Lemma~\ref{a.e.} also holds $\mu_s^i$-a.e.~in~$\mathrm X$ for each
$s\in S$. Having integrated this relation with respect to~$\mu_s^i$
and then summing up the inequalities obtained over all $i\in I$, on
account of~(\ref{xxx}) and~(\ref{wsum}) we get
\begin{align*}
\kappa(\chi+R\boldsymbol{\gamma},R\boldsymbol{\mu}_s)&=\sum_{i\in
I}\,\bigl\langle
W_{\boldsymbol\gamma}^i,\mu_s^i\bigr\rangle\geqslant\sum_{i\in
I}\,\max\Bigl\{\eta_i,\bigl\langle
W_{\boldsymbol\gamma}^i,\gamma^i\bigr\rangle\Bigr\}\geqslant\sum_{i\in
I}\,\eta_i\\{}&=\frac12\Bigl[\|\boldsymbol\gamma\|^2_{\mathcal
E^+(\boldsymbol A)}+G_{\chi}(\boldsymbol A,\boldsymbol
a,g)\Bigr]=\sum_{i\in I}\,\bigl\langle
W_{\boldsymbol\gamma}^i,\gamma^i\bigr\rangle=\kappa(\chi+R\boldsymbol{\gamma},R\boldsymbol{\gamma}),\end{align*}
which in view of the strong (hence, weak) convergence of
$(R\boldsymbol{\mu}_s)_{s\in S}$ to~$R\boldsymbol{\gamma}$
proves~(\ref{etaeta}).
\end{proof}

\begin{corollary}\label{eta3}For any\/ ${\boldsymbol\gamma},{\boldsymbol\gamma}_1\in\mathfrak
E_\chi(\boldsymbol{A},\boldsymbol{a},g)$, it holds that
\[\bigl\langle
W_{{\boldsymbol\gamma}_1}^i,\gamma_1^i\bigr\rangle=\bigl\langle
W_{\boldsymbol\gamma}^i,\gamma^i\bigr\rangle\quad\text{for all \
}i\in I.\]
\end{corollary}

\begin{proof}Indeed, then $\|\boldsymbol\gamma-{\boldsymbol\gamma}_1\|_{\mathcal E^+(\boldsymbol A)}=0$
by assertion~(ii) of Lemma~\ref{lemma:WM}, which in view of
Corollary~\ref{lemma:potequiv} yields $\boldsymbol
W_{\boldsymbol\gamma}=\boldsymbol W_{{\boldsymbol\gamma}_1}$
n.e.~in~$\mathrm X$. Combining this with relation~(\ref{descpot1}),
we get
\begin{equation*}
a_iW_{{\boldsymbol\gamma}_1}^i(x)\geqslant\bigl\langle
W_{\boldsymbol\gamma}^i,\gamma^i\bigr\rangle g(x)\quad\mbox{n.e.~in
\ } A_i,\quad i\in I,\end{equation*} which establishes the corollary
due to the uniqueness statement in
Theorem~\ref{th:descpot1}.
\end{proof}

\section{Proof of Theorem~\ref{th:solvable2}}\label{proof:th:solvable2}

Fix $\tilde{\boldsymbol\lambda}\in\mathfrak S_\chi(\boldsymbol
A,\boldsymbol a,g,CL)$; it exists according to
Theorem~\ref{lemma:exist} with $J=CL=I\setminus\{\ell\}$. In
consequence of Theorem~\ref{th:solvable1},
Theorem~\ref{th:solvable2} will be proved once we establish the
existence of a solution to the (main) $(\boldsymbol A,\boldsymbol
a,g,\chi)$-problem under the hypothesis
\begin{equation}
\label{aelll}a_\ell<\langle g,\tilde{\lambda}^\ell\rangle.
\end{equation}
(Observe that the value on the right does not depend on the choice
of~$\tilde{\boldsymbol\lambda}$, which is clear from
assumption~(\ref{nonint}) in view of
$R\tilde{\boldsymbol\lambda}_1=R\tilde{\boldsymbol\lambda}_2$ for
all
$\tilde{\boldsymbol\lambda}_1,\tilde{\boldsymbol\lambda}_2\in\mathfrak
S_\chi(\boldsymbol A,\boldsymbol a,g,CL)$; see
Corollary~\ref{cor:uniqu}.)

Let $\{K_\ell\}_{A_\ell}$ denote the increasing filtering family of
all compact subsets~$K_\ell$ of~$A_\ell$. For each
$K_\ell\in\{K_\ell\}_{A_\ell}$, define the $(I^+,I^-)$-condenser
$\boldsymbol A_{K_\ell}=\bigl(A^{K_\ell}_i\bigr)_{i\in I}$ by
\begin{equation*}
 \label{exh}A^{K_\ell}_i:=A_i\quad\text{for all \ }i\ne\ell,\qquad
A^{K_\ell}_\ell:=K_\ell.
\end{equation*}

\begin{lemma}\label{lemma:cont}There holds the following statement on continuity:
\begin{equation*}\label{contnew}
G_{\chi}(\boldsymbol{A},\boldsymbol{a},g)=\lim_{K_\ell\uparrow
A_\ell}\,
G_{\chi}(\boldsymbol{A}_{K_\ell},\boldsymbol{a},g).\end{equation*}
\end{lemma}

\begin{proof} Fix
$\boldsymbol\mu\in\mathcal E^+(\boldsymbol{A},\boldsymbol{a},g)$.
For every $K_\ell\in\{K_\ell\}_{A_\ell}$, consider
$\mu^\ell_{K_\ell}:=\mu^\ell|_{K_\ell}$, the trace of~$\mu^\ell$
on~$K_\ell$. Applying \cite[Lemma~1.2.2]{F1}, we then
obtain\footnote{See Section~\ref{sec:main} for the notation
$\boldsymbol\mu_J$.}
\begin{align}
\label{w}\langle g,\mu^\ell\rangle&=\lim_{K_\ell\uparrow A_\ell}\,\langle g,\mu^\ell_{K_\ell}\rangle,\\
\label{www}\kappa(\chi,\mu^\ell)&=\lim_{K_\ell\uparrow A_\ell}\,\kappa(\chi,\mu^\ell_{K_\ell}),\\
\label{ww'}\kappa(\mu^\ell,\mu^\ell)&=\lim_{K_\ell\uparrow A_\ell}\,\kappa(\mu^\ell_{K_\ell},\mu^\ell_{K_\ell}),\\
 \label{ww}\kappa(\mu^\ell,R\boldsymbol\mu_{CL})&=\lim_{K_\ell\uparrow
A_\ell}\,\kappa(\mu^\ell_{K_\ell},R\boldsymbol\mu_{CL}).
\end{align}
Choose a compact set $K^0_\ell\subset A_\ell$ so that $\langle
g,\mu^\ell_{K^0_\ell}\rangle>0$, which is possible due to~(\ref{w}),
and for all $K_\ell\in\{K_\ell\}_{A_\ell}$ that follow~$K^0_\ell$
define
$\hat{\boldsymbol\mu}_{\boldsymbol{A}_{K_\ell}}=\bigl(\hat{\mu}^i_{\boldsymbol{A}_{K_\ell}}\bigr)_{i\in
I}$, where
\begin{equation}\label{hatmu'}\hat{\mu}^i_{\boldsymbol{A}_{K_\ell}}:=\mu^i\quad\text{for all \ }i\ne\ell,\qquad
\hat{\mu}^\ell_{\boldsymbol{A}_{K_\ell}}:=\frac{a_\ell}{\langle
g,\mu_{K_\ell}^\ell\rangle}\,\mu_{K_\ell}^\ell.\end{equation} Since
then
\[R\hat{\boldsymbol\mu}_{\boldsymbol{A}_{K_\ell}}=R\boldsymbol\mu_{CL}-\frac{a_\ell}{\langle
g,\mu_{K_\ell}^\ell\rangle}\,\mu_{K_\ell}^\ell,\] from
Corollary~\ref{lemma3.3} we get
$\hat{\boldsymbol\mu}_{\boldsymbol{A}_{K_\ell}}\in\mathcal
E^+(\boldsymbol{A}_{K_\ell},\boldsymbol a,g)$, and hence, by
(\ref{gmu}) and~(\ref{w})--(\ref{hatmu'}),
\begin{equation*}
\begin{split}
G_\chi(\mu)&=\|R\boldsymbol\mu_{CL}\|^2-
2\kappa(\mu^\ell,R\boldsymbol\mu_{CL})+\|\mu^\ell\|^2+2\kappa(\chi,R\boldsymbol\mu_{CL})-
2\kappa(\chi,\mu^\ell)\\
{}&=\|R\boldsymbol\mu_{CL}\|^2+2\kappa(\chi,R\boldsymbol\mu_{CL})+
\lim_{K_\ell\uparrow
A_\ell}\,\Bigl\{\|\hat{\mu}^\ell_{\boldsymbol{A}_{K_\ell}}\|^2-
2\kappa(\hat{\mu}^\ell_{K_\ell},
R\boldsymbol\mu_{CL})-2\kappa(\chi,\hat{\mu}^\ell_{\boldsymbol{A}_{K_\ell}})\Bigr\}\\
{}&=\lim_{K_\ell\uparrow
A_\ell}\,\Bigl\{\|R\hat{\boldsymbol\mu}_{\boldsymbol{A}_{K_\ell}}\|^2+
2\kappa(\chi,R\hat{\boldsymbol\mu}_{\boldsymbol{A}_{K_\ell}})\Bigr\}\geqslant\lim_{K_\ell\uparrow
A_\ell}\,G_\chi(\boldsymbol{A}_{K_\ell},\boldsymbol a,g),\end{split}
\end{equation*}
which in view of the arbitrary choice of $\boldsymbol\mu\in\mathcal
E^+(\boldsymbol{A},\boldsymbol{a},g)$ proves
\[G_\chi(\boldsymbol{A},\boldsymbol a,g)\geqslant\lim_{K_\ell\uparrow A_\ell}\,G_\chi(\boldsymbol{A}_{K_\ell},
\boldsymbol a,g).\]
Since the converse inequality is obvious, the lemma is
established.
\end{proof}

In the rest of the proof we assume $G_\chi(\boldsymbol
A_{K_\ell},\boldsymbol a,g)<\infty$ for all
$K_\ell\in\{K_\ell\}_{A_\ell}$, which involves no loss of generality
because of condition~(\ref{G:finite}) and Lemma~\ref{lemma:cont}.
Then, according to Theorem~\ref{th:suff}, for every
$\boldsymbol{A}_{K_\ell}$ there exists
$\boldsymbol\lambda_{\boldsymbol{A}_{K_\ell}}\in\mathfrak
S_\chi(\boldsymbol{A}_{K_\ell},\boldsymbol a,g)$, while in
consequence of Lemma~\ref{lemma:cont}, these minimizers form a net
minimizing in the $(\boldsymbol A,\boldsymbol a,g,\chi)$-problem,
i.e.
\begin{equation}\label{last}(\boldsymbol\lambda_{\boldsymbol{A}_{K_\ell}})_{K_\ell\in\{K_\ell\}_{A_\ell}}\in
\mathbb M_\chi(\boldsymbol{A},\boldsymbol a,g).
\end{equation}

Fix a (particular) extremal measure $\boldsymbol\gamma\in\mathfrak
E_\chi(\boldsymbol{A},\boldsymbol{a},g)$ that is a vague cluster
point of the net
$(\boldsymbol\lambda_{\boldsymbol{A}_{K_\ell}})_{K_\ell\in\{K_\ell\}_{A_\ell}}$;
it exists due to inclusion~(\ref{last}) and assertion~(i) of
Lemma~\ref{lemma:WM}.

\begin{lemma}\label{gamma00} For this\/ {\rm(}particular\/{\rm)} extremal measure\/~$\boldsymbol\gamma$, it holds that
\begin{equation}\label{descpot2}
a_iW_{\boldsymbol\gamma}^i(x)\leqslant\bigl\langle
W_{\boldsymbol\gamma}^i,\gamma^i\bigr\rangle g(x)\quad\mbox{for all
\ } x\in S(\gamma^i),\quad i\in I,\end{equation} and therefore, in
consequence of relation\/~{\rm(\ref{descpot1})},
\begin{equation}\label{descpot3}
a_iW_{\boldsymbol\gamma}^i(x)=\bigl\langle
W_{\boldsymbol\gamma}^i,\gamma^i\bigr\rangle g(x)\quad\mbox{n.e.~in
\ } S(\gamma^i),\quad i\in I.\end{equation}\end{lemma}

\begin{proof} Fix $i\in I$ and $x_0\in S(\gamma^i)$. Passing to a subnet if
necessary, assume
$(\boldsymbol\lambda_{\boldsymbol{A}_{K_\ell}})_{K_\ell\in\{K_\ell\}_{A_\ell}}$
to converge vaguely to~$\boldsymbol\gamma$. Then one can find
$\zeta_{K_\ell}\in S(\lambda_{\boldsymbol{A}_{K_\ell}}^i)$ such that
$\zeta_{K_\ell}\to x_0$ as $K_\ell\uparrow A_\ell$.

Note that, under the standing assumptions (cf.~also
Remark~\ref{rem:lower}), \cite[Theorem~7.2]{ZPot2} is applicable to
each of $\boldsymbol\lambda_{\boldsymbol{A}_{K_\ell}}$,
$K_\ell\in\{K_\ell\}_{A_\ell}$. This gives
\begin{equation}\label{nuuu}a_iW_{\boldsymbol\lambda_{\boldsymbol{A}_{K_\ell}}}^i(\zeta_{K_\ell})\leqslant\bigl\langle
W_{\boldsymbol\lambda_{\boldsymbol{A}_{K_\ell}}}^i,\lambda_{\boldsymbol{A}_{K_\ell}}^i\bigr\rangle
g(\zeta_{K_\ell}).\end{equation} Another observation is that,
according to~(\ref{threestars})
with~$(\boldsymbol\lambda_{\boldsymbol{A}_{K_\ell}})_{K_\ell\in\{K_\ell\}_{A_\ell}}$
instead of~$(\boldsymbol\mu_s)_{s\in S}$,
\begin{equation}\label{nu''}\bigl\langle
W_{\boldsymbol\gamma}^i,\gamma^i\bigr\rangle=\lim_{K_\ell\uparrow
A_\ell}\,\bigl\langle
W_{\boldsymbol\lambda_{\boldsymbol{A}_{K_\ell}}}^i,
\lambda_{\boldsymbol{A}_{K_\ell}}^i\bigr\rangle.\end{equation}

Having substituted (\ref{wpopot}) with
$\boldsymbol\mu=\boldsymbol\lambda_{\boldsymbol{A}_{K_\ell}}$ into
the left-hand side of inequality~(\ref{nuuu}), we pass to the limit
as $K_\ell\uparrow A_\ell$. In view of~(\ref{nu''}) and the lower
semicontinuity of $f_i=\alpha_i\kappa(x,\chi)$ on~$A_i$ (see
Remark~\ref{rem:lower}), we then see that assertion~(\ref{descpot2})
will be proved once we establish
\[\kappa(x_0,R\boldsymbol\gamma^-)=
\lim_{K_\ell\uparrow
A_\ell}\,\kappa(\zeta_{K_\ell},R\boldsymbol\lambda_{\boldsymbol{A}_{K_\ell}}^-)\quad\text{and}\quad\kappa(x_0,R\boldsymbol\gamma^+)\leqslant\lim_{K_\ell\uparrow
A_\ell}\,\kappa(\zeta_{K_\ell},R\boldsymbol\lambda_{\boldsymbol{A}_{K_\ell}}^+),\]
but these two follow directly from Lemma~\ref{lemma:vague.cont} and
the lower semicontinuity of the mapping
$(x,\nu)\mapsto\kappa(x,\nu)$ on $\mathrm X\times\mathfrak M^+(X)$
(see~\cite[Lemma~2.2.1]{F1}), respectively.
\end{proof}

Applying Lemma~\ref{lemma:lemma} to $\boldsymbol\gamma$,  we get
$\langle g,\gamma^i\rangle=a_i$ for all $i\ne\ell$, so that
\begin{equation}\label{g000}\boldsymbol\gamma\in\mathcal E^+(\boldsymbol A,\boldsymbol a,g,CL).\end{equation}
We next proceed to show that, due to the additional
requirement~(\ref{aelll}), it also holds
\begin{equation}\label{g0000}\langle g,\gamma^\ell\rangle=a_\ell.\end{equation}

\begin{lemma}\label{0000}It is true that\/ $\langle
 W_{\boldsymbol\gamma}^\ell,\gamma^\ell\rangle\ne0$.\end{lemma}

\begin{proof}Assume, on the contrary, that
\begin{equation}
 \label{wjzero}\langle W_{\boldsymbol\gamma}^\ell,\gamma^\ell\rangle=0.
\end{equation}
Then, according to (\ref{xxx}) with
$\boldsymbol\mu=\boldsymbol\mu_1=\boldsymbol\gamma$,
\begin{equation}
 \label{sumwj}\sum_{j\ne\ell}\,\langle W_{\boldsymbol\gamma}^j,\gamma^j\rangle=\sum_{i\in I}\,\langle W_{\boldsymbol\gamma}^i,\gamma^i\rangle
=\kappa(\chi+R\boldsymbol\gamma,R\boldsymbol\gamma).
\end{equation}
As all the coordinates of the given
$\tilde{\boldsymbol\lambda}\in\mathfrak S_\chi(\boldsymbol
A,\boldsymbol a,g,CL)$ are bounded by Corollary~\ref{min:proj}, for
each $i\in I$ relation~(\ref{descpot1}) holds
$\tilde{\lambda}^i$-a.e. in~$\mathrm X$; that is,
\[a_iW_{\boldsymbol\gamma}^i(x)\geqslant\langle W_{\boldsymbol\gamma}^i,\gamma^i\rangle g(x)\quad
\tilde{\lambda}^i\text{-a.e. in \ }\mathrm X,\quad i\in I.\] Having
integrated this relation, divided by~$a_i$, with respect to
$\tilde{\lambda}^i$ and then summing up the inequalities obtained
over all $i\in I$, on account of equalities~(\ref{wjzero}),
(\ref{sumwj}) and $\langle g,\tilde{\lambda}^i\rangle=a_i$ for all
$i\ne\ell$ we get
\[\sum_{i\in I}\,\langle W_{\boldsymbol\gamma}^i,\tilde{\lambda}^i\rangle\geqslant\sum_{i\in I}\,
\frac{\langle W_{\boldsymbol\gamma}^i,\gamma^i\rangle}{a_i}\langle
g,\tilde{\lambda}^i\rangle= \sum_{j\ne\ell}\,\langle
W_{\boldsymbol\gamma}^j,\gamma^j\rangle=\kappa(\chi+R\boldsymbol\gamma,R\boldsymbol\gamma),\]
which because of (\ref{xxx}) with $\boldsymbol\mu=\boldsymbol\gamma$
and $\boldsymbol\mu_1=\tilde{\boldsymbol\lambda}$ can be rewritten
in the form
\[\kappa(\chi+R\boldsymbol\gamma,R\boldsymbol\gamma)\leqslant
\kappa(\chi+R\boldsymbol\gamma,R\tilde{\boldsymbol\lambda})\] or,
equivalently,
\[\|\chi+R\boldsymbol\gamma\|^2\leqslant\kappa(\chi+R\boldsymbol\gamma,\chi+R\tilde{\boldsymbol\lambda}).\]
Applying the Cauchy--Schwarz inequality and, subsequently,
identity~(\ref{gmu}), we then obtain
\[G_\chi(\boldsymbol\gamma)\leqslant G_\chi(\boldsymbol A,\boldsymbol a,g,CL),\]
which in view of inclusion~(\ref{g000}) yields
$\boldsymbol\gamma\in\mathfrak S_\chi(\boldsymbol A,\boldsymbol
a,g,CL)$.

By relation~(\ref{aelll}) and the observation following it, we thus
get
\[\langle g,\gamma^\ell\rangle>a_\ell.\]
Since this contradicts to relation~(\ref{inequal}), the lemma is
established.
\end{proof}

Now, having integrated (\ref{descpot3}) for $i=\ell$ with respect to
the (bounded) measure~$\gamma^\ell$, on account of Lemma~\ref{0000}
we obtain~(\ref{g0000}), which together with inclusion~(\ref{g000})
shows that, actually, $\boldsymbol\gamma\in\mathcal E^+(\boldsymbol
A,\boldsymbol a,g)$. Combined with Corollary~\ref{IorII}, this
proves that $\boldsymbol\gamma$ is, in fact, a solution to the
$(\boldsymbol A,\boldsymbol a,g,\chi)$-problem, and
Theorem~\ref{th:solvable2} follows.\hfill$\square$

\bigskip\noindent{\bf\small Acknowledgments} {\small The author is very grateful
to Professor W.L.~Wendland for the careful reading of the manuscript
and his suggestions to improve it. A part of this research was done
during the author's visit to the DFG Cluster of Excellence
Simulation Technologies at the University of Stuttgart during
February of 2012, and the author acknowledges this institution for
the support and the excellent working conditions.}


\begin{thebibliography}{00}

\bibitem{A} Aptekarev, A.I., Lysov, V.G.: Systems of Markov functions generated
by graphs and the asymptotics of their Hermite--Pad\'{e}
approximants. Sb. Math. {\bf 201}, 183--234 (2010)

\bibitem{B2}
Bourbaki, N.: Elements of Mathematics, Integration, chapters~1--6.
Springer, Berlin (2004)

\bibitem{Brelo2} Brelot, M.: On Topologies and Boundaries in
Potential Theory. Lectures Notes in Math., vol.~175. Springer,
Berlin (1971)

\bibitem{Car} Cartan, H.: Th\'eorie du potentiel Newtonien:
\'energie, capacit\'e, suites de potentiels. Bull. Soc. Math. Fr.
{\bf 73}, 74--106 (1945)

\bibitem{D1} Deny, J.: Les potentiels d'\'energie finite. Acta Math. {\bf 82}, 107--183 (1950)

\bibitem{D2} Deny, J.: Sur la d\'{e}finition de
l'\'{e}nergie en th\'{e}orie du potentiel. Ann. Inst. Fourier
Grenoble {\bf 2}, 83--99 (1950)

\bibitem{E1}
Edwards, R.: Cartan's balayage theory for hyperbolic Riemann
surfaces. Ann. Inst. Fourier {\bf 8}, 263--272 (1958)

\bibitem{E2}
Edwards, R.: Functional Analysis. Theory and Applications. Holt.
Rinehart and Winston, New York (1965)

\bibitem{Fr}  Frostman, O.: Potentiel d'\'{e}quilibre et capacit\'{e} des ensembles avec
quelques applications \`{a} la th\'{e}orie des fonctions. Comm.
S\'{e}m. Math. Univ. Lund~{\bf 3}, 1--118 (1935)

\bibitem{F1} Fuglede, B.: On the theory of potentials in
locally compact spaces. Acta Math.~{\bf 103}, 139--215 (1960)

\bibitem{F2} Fuglede, B.: Caract\'erisation des noyaux consistants
en th\'eorie du potentiel. Comptes Rendus~{\bf 255}, 241--243 (1962)

\bibitem{F3} Fuglede, B.: Asymptotic paths for subharmonic functions
and polygonal connectedness of fine domains. In: Lectures Notes in
Math., vol.~814, 97--115. Springer, Berlin (1980)

\bibitem{Gauss} Gauss, C.F.: Allgemeine Lehrs\"atze in Beziehung auf die im
verkehrten Verh\"altnisse des Quadrats der Entfernung wirkenden
Anziehungs-- und Absto{\ss}ungs--Kr\"afte (1839). Werke {\bf 5},
197--244 (1867)

\bibitem{GR0} Gonchar, A.A., Rakhmanov, E.A.: On convergence
of simultaneous Pad\'{e} approximants for systems of functions of
Markov type. Proc. Steklov Inst. Math. {\bf 157}, 31--50 (1983)

\bibitem{GR}
Gonchar, A.A., Rakhmanov, E.A.: On the equilibrium problem for
vector potentials. Russ. Math. Surv. {\bf 40}(4), 183--184 (1985)

\bibitem{GR1} Gonchar, A.A., Rakhmanov, E.A.: Equilibrium
measure and the distribution of zeros of extremal polynomials. Math.
USSR-Sb. {\bf 53}, 119--130 (1986)

\bibitem{GRS} Gonchar, A.A., Rakhmanov, E.A., Sorokin, V.N.: Hermite--Pad\'{e}
approximants for systems of Markov-type functions. Sb. Math. {\bf
188}, 671--696 (1997)

\bibitem{HWZ2}
Harbrecht, H., Wendland, W.L., Zorii, N.: Riesz minimal energy
problems on $C^{k-1,1}$-manifolds. Preprint Series Stuttgart
Research Centre for Simulation Technology (2012)

\bibitem{Hayman2}
Hayman, W.K.: Subharmonic Functions, vol.~2.  Academic Press, London
(1989)

\bibitem{HK}
Hayman, W.K., Kennedy, P.B.: Subharmonic Functions, vol.~1. Academic
Press, London (1976)

\bibitem{K}
Kelley, J.L.: General Topology. Princeton, New York (1957)

\bibitem{L}
Landkof, N.S.: Foundations of Modern Potential Theory. Springer,
Berlin (1972)

\bibitem{MaS} Mhaskar, H.N., Saff, E.B.: Extremal problems
for polynomials with exponential weights. Trans. Amer. Math. Soc.
{\bf 285}, 204--234 (1984)

\bibitem{MS} Moore, E.H., Smith, H.L.: A general theory of
limits. Amer. J. Math. {\bf 44}, 102--121 (1922)

\bibitem{NS} Nikishin, E.M., Sorokin, V.N.: Rational Approximations and
Orthogonality. Translations of Mathematical Monographs, vol.~44.
Amer. Math. Soc., Providence (1991)

\bibitem{OWZ}
Of, G., Wendland, W.L., Zorii, N.: On the numerical solution of
minimal energy problems. Complex Variables and Elliptic Equations
\textbf{55}, 991--1012 (2010)

\bibitem{O}
Ohtsuka, M.: On potentials in locally compact spaces.
J.~Sci.~Hiroshima Univ. Ser.~A-1 {\bf 25}, 135--352 (1961)

\bibitem{ST} Saff, E.B., Totik, V.: Logarithmic Potentials
with External Fields. Sprin\-ger, Berlin (1997)

\bibitem{Z0} Zorii, N.: An extremal problem on the minimum of energy for space
condensers. Ukr. Math.~J. {\bf 38}, 365--370 (1986)

\bibitem{Z1} Zorii, N.: A problem of minimum energy for space
condensers and Riesz kernels.  Ukr. Math.~J. \textbf{41}, 29--36
(1989)

\bibitem{Z2} Zorii, N.: A noncompact variational problem in
Riesz potential theory.~I.  Ukr. Math.~J. {\bf 47}, 1541--1553
(1995)

\bibitem{Z5a} Zorii, N.: Equilibrium potentials
with external fields. Ukr. Math.~J. {\bf 55}, 1423--1444 (2003)

\bibitem{Z5} Zorii, N.: Equilibrium problems for potentials
with external fields. Ukr. Math.~J. {\bf 55}, 1588--1618 (2003)

\bibitem{Z6} Zorii, N.: Necessary and sufficient conditions for
the solvability of the Gauss variational problem. Ukr. Math.~J. {\bf
57}, 70--99 (2005)

\bibitem{Z-Pot} Zorii, N.: Interior capacities of condensers in
locally compact spaces. Potential Anal. {\bf 35}, 103--143 (2011),
DOI:10.1007/s11118-010-9204-y

\bibitem{ZPot2} Zorii, N.: Equilibrium problems for infinite dimensional vector potentials
with external fields. Potential Anal., DOI:10.1007/s11118-012-9279-8


\end{thebibliography}
\end{document}